\documentclass[english]{amsart}

\usepackage{amssymb}
\usepackage{amsmath}
\usepackage{amsfonts}
\usepackage{paralist}
\usepackage{pst-all}
\usepackage{pst-poly}
\usepackage{graphics}
\usepackage{graphicx}
\usepackage{epsfig}
\usepackage{color}
\usepackage[all]{xy}
\usepackage{subfigure}
\xyoption{arc}

\usepackage[colorlinks=true]{hyperref}
\hypersetup{urlcolor=blue, citecolor=red}

\newcommand{\rr}{\mathbb{R}}
\newcommand{\zz}{\mathbb{Z}}

\newcommand{\nn}{\mathbb{N}}

\newcommand{\dd}{\mathbb{D}}
\newcommand{\sss}{\mathbb{S}}

\newcommand{\groupoid}[2]{\xymatrix@C=16pt{ #1 \ar@<1.5pt>[r]^{\scriptscriptstyle  \ \ \alpha}
\ar@<-1.5pt>[r]_{\scriptscriptstyle \ \, \beta} & #2}}
\newcommand{\F}{\mathcal{F}} 
\newcommand{\R}{\mathcal{R}}
\newcommand{\scirc}{{\scriptstyle \circ}}

\newcommand{\matrice}[4]{ ( \, 
\raisebox{-0.8ex}{$\stackrel{^{\scriptstyle #1}}{\scriptstyle #2}$
$\stackrel{^{\scriptstyle #3}}{\scriptstyle #4}$} \, ) }

\newtheorem{theorem}{Theorem}[section]

\newtheorem{lemma}[theorem]{Lemma}
\newtheorem{proposition}[theorem]{Proposition}

\theoremstyle{definition}
\newtheorem{definition}[theorem]{Definition}
\newtheorem{remark}[theorem]{Remark}
\newtheorem{remarks}[theorem]{Remarks}

\begin{document}

\title[Minimal F{\o}lner foliations are amenable]
      {Minimal F{\o}lner foliations are amenable}

\author[Fernando Alcalde Cuesta]{Fernando Alcalde Cuesta} 
\address{Departamento de Xeometr\'{\i}a e Topolox\'{\i}a \\
Universidade de Santiago de Compostela \\  E-15782 Santiago de Compostela (Spain)}
\email{fernando.alcalde@usc.es}

\author[Ana Rechtman]{Ana Rechtman}
\address{Department of Mathematics \\ University of Illinois at Chicago, 322 SEO \\ 851 S. Morgan Street, Chicago IL 60607-7045, USA}
\email{rechtman@math.uic.edu}
\date{}




\begin{abstract}
For finitely generated groups, amenability and F\o lner properties are equivalent.  However, contrary to a widespread idea, Kaimanovich showed that F\o lner condition does not imply amenability for discrete measured equiva\-lence relations. In this paper, we exhibit two examples of $C^\infty$ foliations of closed manifolds that are F\o lner and non amenable  with respect to a finite transverse invariant measure and a transverse invariant volume, respectively.
We also prove the equivalence between the two notions when the foliation is minimal, that is all the leaves are dense, giving a positive answer to a question of Kaimanovich. The equivalence is stated with respect to transverse \mbox{invariant} measures or some tangentially smooth measures. 
The latter include harmonic measures, and in this case the F{\o}lner condition has to be replaced by $\eta$-F{\o}lner (where the usual volume is modified by the modular form $\eta$ of the measure).

\end{abstract}

\maketitle

\section{Introduction}
\label{introduccion}

Discrete equivalence relations provide a natural approach to the study of foliations on compact manifolds:  the leaves induce a discrete equivalence relation on any total transversal. We can also see this equivalence relation as the orbit equivalence relation of the holonomy pseudogroup, which is of finite type in the compact case. Some properties of the foliation can be appreciated in this discrete setting, and do not depend upon the choices that we made. This is the case for the {\it amenability} and {\it F{\o}lner} properties, the two main notions that we will study in this paper.
\medskip 

This two notions are motivated by the corresponding ones for finitely generated groups. 
A finitely generated group is {\it amenable} if and only if it is {\it F{\o}lner}. Let us start by recalling the definitions for groups. A finitely generated group $G$ is said to be {\it amenable} if there is an {\it invariant mean}, that is a positive linear functional on the Banach space $l^\infty(G)$ which maps the constant function $1$ to $1$, and is translation invariant. E. F{\o}lner showed that $G$ is amenable if and only if
$$
\inf_{E}\frac{|\partial E|}{|E|}=0,
$$
where $|\cdot|$ denotes the cardinality of a set, $E\subset G$ are finite subsets, and $\partial E$ is the boundary of $E$ with respect to a given set of generators of $G$. A group satisfying the latter condition is  said to be {\it F{\o}lner}. 
\medskip 

Both definitions can be easily stated for a compact foliated manifold
equipped with a transverse invariant measure, using the induced
equivalence relation. {\it Amena\-bility} is the property of having (in
a measurable way) a mean on almost all the equivalence classes, with
respect to the given measure (see \S~\ref{promediable}). On the other hand, each equivalence class may be realized as the set of vertices of a graph, which is naturally quasi-isometric to the corresponding leaf. An equivalence class is {\it F{\o}lner} if there are finite subgraphs $A$ with arbitrary small isoperimetric ratio $\frac{|\partial A|}{|A|}$, where $\partial A$ is the boundary of $A$. We say that a foliation is {\it F{\o}lner} if almost all the equivalence classes on a total transversal are F{\o}lner. 
\medskip 

In this paper we are concerned with the relation between this two concepts for compact foliated manifolds. In 1983, R. Brooks stated without proving (see example-theorem 4.3 of \cite{broo}): {\it Let $\mathcal{F}$ be a foliation with invariant measure $\mu$. If $\mu$-almost all leaves are F{\o}lner, $\mathcal{F}$ is amenable with respect to $\mu$.} One of the aims of the present paper is to show that amenability cannot be deduced from the condition of having F{\o}lner leaves, thus disproving Brooks' statement.  This will be done with two examples of $C^\infty$ foliations of closed manifolds that are F{\o}lner and non-amenable. Both examples can be made real analytic. It is important to say that in 2001, V. A. Kaimanovich constructed several examples of discrete equivalence relations that are F{\o}lner and non-amenable \cite{kaim}. In the same paper, he gave an example of a $C^\infty$ foliation satisfying the same properties with respect to a transverse invariant measure, that is not finite. The relevance in the two examples presented in this paper are the measures: in the first one, presented in \S~\ref{seccionejemploreeb}, the measure in consideration is a finite transverse invariant measure; and in the second one, \S~\ref{seccionejemplowilson}, it is a transverse invariant volume. On the other hand, in 1985 it was proved by Y. Carri\`ere and \'E. Ghys that an amenable foliation, with respect to a transverse invariant measure, is F{\o}lner \cite{cagh}.
\medskip 

The role of the measure considered is crucial. In this paper we will
study the amenability and F{\o}lner properties with respect to either
transverse invariant measures, or {\it tangentially smooth
  measures}. The latter are measures on the ambient manifold that are
smooth along the leaves, and will be introduced in \S~\ref{medidas}. To give an example, we can say that harmonic measures and smooth measures on the manifold are tangentially smooth. As we will explain in \S~\ref{seccionfolner}, we will have to adapt the definition of F{\o}lner leaf for a tangentially smooth measure $\mu$: we will define the concept of $\eta$-F{\o}lner leaf, in the same spirit that F{\o}lner, but using a modified tangent metric. The modification is done with the modular form $\eta$  of the measure $\mu$, as explained in \S~\ref{seccionfolner}.
\medskip 

In \cite{kaim}, Kaimanovich asked if the {\it minimality} of the foliation guarantees the equivalence between the two notions. Here minimal means that all the leaves are dense. 
The main objective of this paper is to present a  proof of the following theorems:

\begin{theorem}
\label{teoremamedidainvariante}
Let $\mathcal{F}$ be a minimal foliation of a compact manifold $M$, and $\nu$ a transverse invariant measure. Assume that $\mathcal{F}$ has no essential holonomy, then $\mathcal{F}$ is amenable with respect to $\nu$ if and only if $\nu$-almost all leaves are F{\o}lner.
\end{theorem}

\begin{theorem}
\label{teoremamedidalisa} 
Let $\mathcal{F}$ be a minimal foliation of a compact manifold $M$,
and $\mu$ a tangentially smooth measure. Assume that $\mathcal{F}$ has
no essential holonomy and that the modified volume of the plaques is bounded.
Then $\mathcal{F}$ is amenable with respect to $\mu$ if and only if $\mu$-almost all leaves are $\eta$-F{\o}lner.\end{theorem}

We say that a foliation has no essential holonomy if almost all the
leaves have trivial holonomy. Under this hypothesis the modular form $\eta$ 
admits a primitive $\log h$. The function $h$ is defined almost everywhere and
is known as the density function. We will use a normalization of this
function (relative to the choice of a total transversal) to modify the
volume on the leaves (details of this definitions are in \S~\ref{medidas}).
Without going into more detail, we will like to point out that the normalized density
function of a harmonic measure is always bounded.
\medskip

Theorem \ref{teoremamedidainvariante} appeared in the
 Ph.D. thesis of the second author \cite{these}. This theorem is a particular case of theorem
\ref{teoremamedidalisa}: harmonic measures are tangentially smooth and
a transverse invariant measure combined with the Riemannian
volume on the leaves forms a harmonic measure (such a harmonic measure
is called  {\it completely invariant}), 
and has normalized density function equal to one.
The implication in the second theorem stating that an amenable foliation is F{\o}lner is a generalization of Carri\`ere and Ghys' result. They proved this implication for transverse invariant measures. We will rewrite their proof in this new setting in \S~\ref{demostracioncarriereghys}. Even if the proof of the other implication in the two theorems is very similar, we present them separately since we believe that in the one of theorem \ref{teoremamedidalisa} the technicalities hide the main ideas. For this implication, we will use a theorem by D. Cass \cite{cass} that describes minimal leaves.
\medskip 

The paper is divided in three sections: the first one contains the
definitions and concepts we will use; in the second one we describe
the two examples of non-amena\-ble F{\o}lner foliations; and the third
one contains the proofs of the main theorems.
\medskip 

The second author will like to thank her Ph.D. thesis advisor \'Etienne Ghys for all his patience and comments during the preparation of this work; and Mexico, that has sponsored her Ph.D. studies through the scholarship program of the Consejo Nacional de Ciencia y Tecnolog\'ia (CONACyT). This work was partially supported by ANR-06-BLAN-0030 in France and Xunta de Galicia INCITE08E1R207051ES in Spain.

\section{Preliminaries and general definitions}
\label{notaciondef}

For the purposes of this paper we will always consider a foliation with a measure,  which may be a transverse quasi-invariant measure or a tangentially smooth measure on the ambient space. We will always consider a compact foliated manifold $(M,\mathcal{F})$ of class $C^r$, for $2\leq r\leq \infty$ or $r=\omega$, endowed with a Riemannian metric $g$ that induces a quasi-isometric class of Riemannian metrics on the leaves. Let us consider a foliated atlas 
$\mathcal{A}=\{(U_i,\phi_i)\}_{i\in I}$ where $\phi_i:U_i\to P_i\times T_i$ is a map from an open subset of $M$ to the product of open discs in $\mathbb{R}^d$ and $\mathbb{R}^{m-d}$. Let $\pi_i:U_i\to T_i$ be the natural projection onto the {\it local transversal}  $T_i$ whose fibres are the {\it plaques} of $U_i$. If $U_i \cap U_j \neq \emptyset$, the change of charts 
$\Phi_{ij} = \phi_j \scirc \phi_i^{-1}  : \phi_i(U_i\cap U_j) \to \phi_j(U_i\cap U_j)$ 
is given by
$$
 \Phi_{ij}(x,y) = (\psi_{ij}^y(x), \gamma_{ij}(y))
$$
where $\gamma_{ij}$ is an $C^r$-diffeomorphism between open subsets of $T_i$ and  $T_j$, and 
$\psi_{ij}^y$ is a $C^r$-diffeomorphism depending continuously on $y$ in the $C^r$-topology. We will need the following additional conditions:
\smallskip 

 \noindent
 --  the cover $\mathcal{U} = \{U_i\}_{i\in I}$ is locally finite,
\smallskip 

 \noindent 
 -- for each $i \in I$,  $\overline{U}_i$ is a compact subset of a  foliated chart (not necessarily belonging to $\mathcal{U}$), so the plaques satisfy that their $d$-volume and the $(d-1)$-volume of their boundaries is uniformly bounded (by above and below);
\smallskip 

 \noindent 
 -- if $U_i \cap U_j\neq \emptyset$, there exists a foliated chart containing $\overline{U_i \cap U_j}$ and then each plaque of $U_i$ intersects at most one plaque of $U_j$.
\smallskip 

It is clear that $T=\amalg_{i\in I}T_i$ is a total transversal which intersects every leaf. An equivalence relation $\R$ is naturally defined on $T$:  two points $x,y\in T$ are equivalent if and only if they belong to the same leaf of $\mathcal{F}$.  

\begin{definition}
Let $\R$ be an equivalence relation on a standard Borel space $(T, \mathcal{B})$. We have the following definitions:
\smallskip 

\noindent
--  $\R$ is {\em discrete} if every equivalence class $\R[x]$ is at most countable. 
\smallskip 

\noindent
-- $\R$ is {\em measurable} if its graph is a Borel subset of  $T\times T$.
\smallskip 

\noindent
-- A measure $\nu$ on $(T,\mathcal{B})$ is said to be {\em quasi-invariant} for $\R$ 
if for every Borel set $B\in \mathcal{B}$ with $\nu(B)=0$, the saturation of $B$ is also of measure zero. 
\smallskip 

\noindent 
-- If $\R$ and $\nu$ are as above, 
we say that $\R$ is a {\em discrete measured equivalence relation} on $(T, \mathcal{B}, \nu)$.
\end{definition}

The equivalence relation $\R$ induced by $\mathcal{F}$ on $T$ is
discrete and measured. From another point of view, we can see $\R$ as
the orbit equivalence relation defined by the natural action of the
pseudogroup $\Gamma$, called the {\em holonomy pseudogroup of $\F$}, generated by the local diffeomorphisms $\gamma_{ij}$.
It is important to emphasize that $\{\gamma_{ij}\}$ forms a finite generating set,
that we will call $\Gamma_1$. Notice that we
can visualize each equivalence class $\R[y] = \Gamma(y)$ as a graph:
the vertices are the points in $\R[y]$ and 
the edges correspond to the generators.
We can define a graph metric
$$
d_\Gamma(z,z^\prime)=\min_n\{\exists g \in \Gamma_n |g(z)=z^\prime\},
$$ 
where $\Gamma_n$
are the elements that can be expressed
as words of length at most $n$ in terms of $\Gamma_1$.  
A transverse invariant measure of $\mathcal{F}$, that is a measure on $T$
invariant under the action of $\Gamma$,  gives us an example of a quasi-invariant measure for $\R$.

\subsection{Measures}
\label{medidas}

In the latter paragraph we discussed invariant and quasi-invariant measures. Here we are going to discuss a little bit further about quasi-invariant measures and to introduce {\it tangentially smooth measures}. Contrary to the first two cases, these measures are globally defined. Harmonic measures (introduced by L. Garnett in \cite{garn}) are an example of tangentially smooth measures. An interesting fact about harmonic measures is that they always exist, which is not the case for transversal invariant measures.
\medskip 

Consider the discrete equivalence relation $\R$, induced by a
foliation, on a total transversal $T$ endowed with a quasi-invariant
measure $\nu$. Integrating the counting measures on the fibers of the
left projection $(y,z)\mapsto y$ from $\R$ to $T$ with respect to
$\nu$, gives the {\it left counting measure}
$d\widetilde{\nu}(y,z)=d\nu(y)$.  For the right projection,
we get  the {\it right counting measure} $d\widetilde{\nu}^{-1}(y,z)=d\widetilde{\nu}(z,y)=d\nu(z)$.
Then, $\nu$ is quasi-invariant if and only if $\widetilde{\nu}$ and $\widetilde{\nu}^{-1}$ are equivalent, in which case the Radon-Nikodym derivative
$$
\delta(y,z)=\frac{d\widetilde{\nu}}{d\widetilde{\nu}^{-1}}(y,z)
$$
is called the {\it Radon-Nikodym cocycle} of $(T, \nu,\R)$. 
Finally, let $|\cdot|_y$ be the measure on the equivalence class of $y$ defined as $|z|_y=\delta(z,y)$.
\medskip 

Let us now study global measures. Consider a regular Borel measure $\mu$ on the mani\-fold $M$. Using the foliated atlas, we can give a local decomposition $\mu=\int \lambda_i^yd\nu_i(y)$ on each $U_i$, where $\lambda_i^y$ is a measure on the plaques and $\nu_i$ a measure on $T_i$.
  
\begin{definition}[Tangentially smooth measure]
\label{definiciontangencialmentelisa}
A measure $\mu$ on $(M,\mathcal{F})$ is {\em tangen\-tially smooth} if for every $i\in I$ and $\nu_i$-almost every $y\in T_i$, the measures $\lambda_i^y$ are absolutely continuous with respect to the Riemannian volume $d\mbox{vol}^y$, and the density functions
$$
h_i(x,y)=\frac{d\lambda_i^y}{d\mbox{vol}^y}(x,y),
$$
are smooth functions of class $C^{r-1}$ on the plaques.
\end{definition}

\noindent Observe that the functions $h_i$ are measurable in the transverse direction. We could change $C^\infty$ by $C^r$, for $r\geq 2$, and the results in this paper will still be valid. 
In the intersection of two  foliated charts $U_i$ and $U_j$, we have two local decompositions of the measure $\mu$. Indeed, if $U_i \cap U_j \neq \emptyset$, 
 we have that
 $$
 \mu|_{U_i\cap U_j}=\int \lambda_i^yd\nu_i(y)=\int \lambda_j^yd\nu_j(y).
 $$
Thus pushing the measure $\nu_j$ with the holonomy diffeomorphism $\gamma_{ji}$ we deduce that
$$
\delta_{ij}(y)  =  d((\gamma_{ji})_*\nu_j) / d\nu_i(y) = 
 \frac{h_i(x,y)}{h_j(\psi^y_{ij}(x), \gamma_{ij}(y))}.
 $$
Then the functions $h_i$ verify that $\log h_i-\log h_j=\log \delta_{ij}$ on $U_i\cap U_j$. Since $\delta_{ij}$ is a function on $T_i$, we have that $d_\mathcal{F}\log h_i=d_\mathcal{F}\log h_j$,
where $d_\mathcal{F}$ is the derivative along the leaves of $\mathcal{F}$. Then $\eta= d_\mathcal{F}\log h_i$ is a well defined foliated 1-form of class $C^{r-2}$ along the leaves and measurable in the transverse direction. 

\begin{definition}[Modular form]
 \label{definicionformamodular}
 The 1-form $\eta$ is the {\em modular form} of $ \mu$.
 \end{definition}
  
The modular form measures the transverse measure distortion under the holonomy. If $\sigma$ is a path in a leaf going from a point $p$ to a point $q$, then $\exp(\int_\sigma \eta)$ is the distortion of the transverse measure by the holonomy transformation along $\sigma$. Since the $\delta_{ij}$ are functions defined on $T$, the functions $h_i$ defined on the plaques match in the intersections modulo multiplication by a constant. Thus they define primitives $\log \tilde{h}$ of the induced $1$-form
$\tilde{\eta}$ on the holonomy covering $\tilde{L}$ of each leaf $L$.
If $\mathcal{F}$ has no essential holonomy, the functions $h_i$ can be glued together to get a measurable function $h$ on $M$, such that $\log h$ is a primitive of $\eta$ almost everywhere. 

\begin{remarks} 

\noindent
-- Definition \ref{definiciontangencialmentelisa} is a little more restrictive than 
A. Candel's definition (corollary 5.3 of \cite{cand}). However, harmonic measures are tangentially smooth in our sense. A harmonic measure $\mu$ is completely invariant if and only if $\eta=0$ (we refer to corollary 5.5 of \cite{cand}).
\medskip

\noindent
-- The homotopy groupoid $\Pi_1(\F)$ and the holonomy groupoid $Hol(\F)$
can be interpreted as the larger and smaller unwrapping of $\F$ so that the range projections  $\beta$ are developing maps of the foliated structure (see \cite{phil}). If we denote by $\alpha$ the source projections, this means that  the lifted foliations $\widehat{\F} = \alpha^\ast \F$ on $\Pi_1(\F)$ and $\widetilde{\F} = \alpha^\ast \F$ on $Hol(\F)$ are defined by the $C^r$ submersions $\beta :  \Pi_1(\F) \to M$ and $\beta : Hol(\F) \to M$, respectively. Then $\eta$ lifts  to a pair of foliated $1$-forms 
$\hat{\eta} \in \Omega^1(\widehat{\F})$ and $\tilde{\eta} \in \Omega^1(\widetilde{\F})$.  The first one $\hat{\eta}$ admits a global primitive $\log \hat{h}$ because the first group of foliated cohomology is trivial (we refer to corollary 6.14 of \cite{alcaldehector}). 
Then it is easy to see that this primitive induces on $Hol(\F)$ a global primitive $\log \tilde{h}$ of $\tilde{\eta}$, which puts together all the local primitives described above.
\end{remarks}

\subsection{F{\o}lner foliations} \label{seccionfolner}

Traditionally a graph (with bounded geometry), as for example the equivalence class $\R[x]$ of an equivalence relation induced by a foliation of a compact manifold, is {\it F{\o}lner} if there exist finite subsets $A$ of vertices with arbitrarily small isoperimetric ratio $\frac{|\partial A|}{|A|}$, where $\partial A$ is the set of edges having exactly one end point in $A$ (or the set of  these end points  in $A$), see \cite{goodmanplante} and \cite{kaimanovichpotential}  for the original definitions. The notion of a {\it $\delta$-F{\o}lner equivalence class}, given below, was introduced by Kaimanovich in \cite{kaimcomptes}. In this section we are going to introduce a continuous analogue to his definition: {\it $\eta$-F{\o}lner leaf}. Since the modular form $\eta$ has a primitive on each leaf without holonomy, we have to restrain ourselves to such leaves, that form a residual set. We will say that a foliation is {\it F{\o
 }lner}  (respectively, {\it $\eta$-F{\o}lner}) if almost every leaf is F{\o}lner (respectively, $\eta$-F{\o}lner).
Let us start by stating Kaimanovich's definition in terms of
foliations, observe that the definition does not depends upon the point
$y$.

\begin{definition}
\label{definciondetalfolner}
Let $\mathcal{F}$ be a foliation of a compact manifold $M$.  Let $\R$ be the induced equivalence relation on a total transversal $T$, and assume that $\nu$ is a quasi-invariant measure whose Radon-Nikodym derivative is $\delta(y,z)$. We say that the equivalence class $\R[y]$ is {\em $\delta$-F{\o}lner} if there exists a sequence of finite sets $\{A_n\}_n\subset \R[y]$ such that
$|\partial A_n|_y / |A_n|_y \to 0$,
as $n\to \infty$. 
\end{definition}

Let us pass to the continuous analogue of this notion. Consider the
foliated space $(M,\mathcal{F})$ with a tangentially smooth measure
$\mu$. On a leaf without holonomy  $L_y$ passing through $y \in T$ we can define a leafwise
volume form $d\mbox{vol}_h$ as the volume of the leafwise Riemannian
metric $g$ multiplied by the normalized density function $h / h(y)$. 
We will call this volume the {\it modified volume} and $(h / h(y)) g$ the {\it modified metric} on the leaf $L_y$. 
\
\begin{definition}
\label{definiconfolner}
Let $(M, \mathcal{F})$ be a compact foliated manifold, endowed with a tangentially smooth measure $\mu$. Let $L_y$ be a leaf without holonomy.  Then $L_y$ is {\em $\eta$-F{\o}lner}
if there exists a sequence of compact domains with boundary  $V_n \subset L_y$ such that
$$
\frac{\mbox{area}_h(\partial V_n)}{\mbox{vol}_h(V_n)}\to 0,
$$
as $n\to \infty$. Here $\mbox{area}_h$ denotes the $(d-1)$-volume and $\mbox{vol}_h$ the $d$-volume with respect to the modified metric, and $d$ is the dimension of $\mathcal{F}$.
\end{definition}

\begin{remarks}   \label{hacotada}

\noindent
-- In general, the notion of $\eta$-F\o lner can be applied to the leaves of $\widetilde{\mathcal F}$, that is the holonomy covers of the leaves of $\mathcal F$.
\medskip

\noindent
--  In the case where $\mu$ is a completely invariant harmonic measure, the function $h$ is constant and thus the modified volume and the Riemannian volume coincide. Hence, we recover  the common definition of F{\o}lner leaf. In this situation, the assum\-ption that $M$ is compact together with
the condition that the volume of the plaques and the area of their boundaries are uniformly bounded implies that {\em for every $y\in T$, the leaf $L_y$ passing through $y$ is F{\o}lner if and only if the corresponding graph $\R[y]$ is F{\o}lner.}
\medskip

\noindent
--  When the measure $\mu$ is a harmonic measure (and assuming that
the ambient manifold is compact) we have that the modular form $\eta$
is bounded, we refer to lemma 4.19 on page 116 of B. Deroin
Ph.D. thesis \cite{derointesis}. The result is based on the
Harnack inequality. In fact, the density
function $h$ is harmonic and thus inside a distinguished open set we
have that there exists a constant $C>1$ such that
$$
\frac{1}{C}\leq \frac{h(q)}{h(p)}\leq C,
$$
for all $p \in M $ and for all $q \in L_p$. Hence, the function $\log h$ is uniformly tangentially Lipschtiz, {\it i.e.} $| \log h(p) - log h(q) | \leq C d(p,q),$
where $d(p,q)$ the distance between $p$ and $q$ in $L_p$. This implies that
the modified volume of the plaques and the modified area of their
boundaries remain uniformly bounded.
Therefore, {\em the leaf $L_y$ is $\eta$-F{\o}lner if and only if the graph $\R[y]$ is $\delta$-F{\o}lner}.
\end{remarks}

\subsection{Amenable foliations}
\label{promediable} 

In this section we will recall what amenability is in the context of foliations and study some equivalent notions. A foliation $\mathcal{F}$ of a compact manifold with a transverse quasi-invariant measure $\nu$ is usually said to be {\it amenable}
if the equivalence relation $\R$, on $(T,\nu)$, is $\nu$-amenable. This definition is independent of the choices we made to define the equivalence relation, and was introduced by R. J. Zimmer in \cite{zimm}. 

\begin{definition}
\label{definicionpromediable}
Consider a standard Borel measure space $(T, \mathcal{B},\nu)$. A discrete measured equivalence relation $\R$ on $T$ is {\em amenable} (or {\em $\nu$-amenable}) if 
for $\nu$-almost every $y\in T$ there is a {\em mean} on $\R[y]$ ({\it  i.e.} 
a  linear map
$m_y:L^\infty(R[y])\to \rr$
such that $m_y(f)\geq 0$ for $f\geq 0$ and $m_y(1)=1$) so that the function $y\mapsto m_y$ on $\R$ is
\smallskip 

\noindent
-- {\em measurable}, in the sense that for a measurable function $\tilde{f}$ defined on the 
graph of $\R$, the function defined by $f(y)=m_y(\tilde{f}(y,\cdot))$ is measurable;
\smallskip 

\noindent
-- {\em invariant}, in the sense that $m_y=m_z$ for all $z\in \R[y]$.
\end{definition}

In the following theorem we summarize some equivalent notions of amenability of discrete measured equivalence relations, and thus of a foliation.

\begin{theorem}\label{teoremapromediable}
Let $\R$ be a discrete measured equivalence relation on $(T,\mathcal{B},\nu)$. The following conditions are equivalent:
\smallskip 

\noindent
(i) $\R$ is amenable.
\smallskip 

\noindent
(ii) There exist sequences of probability measures 
$\{ \pi_n^y \}_{y\in T, n\in \nn}$ on $\R[y]$, such that 
$$
\|\pi_n^y-\pi_n^z\|\to 0
$$ 
for  $\widetilde{\nu}$-almost all $(y,z)\in \R$. Here $\|\cdot\|$ is the norm in the space of probability measures on an equivalence class. The map $y \mapsto  \pi_n^y$ is measurable for all $n$, in the same sense as in definition \ref{definicionpromediable}. We will call this Reiter's criterion.
\smallskip 

\noindent
(iii) There exist sequences of probability measures $\{ \pi_n^y \}_{y\in T, n\in \nn}$ on $\R[y]$ such that 
$$
 \int_{\mbox{Dom($\gamma$)}} \|\pi_n^y-\pi_n^{\gamma(y)}\|d\nu(y)\to 0
$$ 
for any partial transformation $\gamma$ of $T$ ({\it  i.e.}  a Borel isomorphism between Borel subsets of $T$ whose graph is contained in $\R$). The map $y \mapsto \pi_n^y$ is measurable for all $n$.
\smallskip 

\noindent
(iv) $\R$ is {\it hyperfinite}, that is there exists an increasing  sequence of finite measured equivalence relations $\R_n$ on $(T, \mathcal{B},\nu)$ such that $\R[y]=\bigcup \R_n[y]$.
\end{theorem}

The equivalence between $(i)$ and $(iv)$ was established by A. Connes, J. Feldman and B. Weiss in \cite{cofw}. They proved that an equivalence relation is amenable if and only if it is generated by an action of $\zz$. Reiter's criterion, both $(ii)$ and $(iii)$ where introduced by Kaimanovich. We refer to \cite{kaimcomptes} and \cite{kaim}.
\medskip 

Let us place ourselves in the general situation of a compact foliated manifold $(M, \mathcal{F})$ with a tangentially smooth measure $\mu$. We have a natural equivalence relation $\R_M$ on $M$ whose equivalence classes are the leaves of $\F$. Therefore, each fibre 
$\R_M^p = \{ (p,q) \in M \times M|  q \in L_p \}$ identifies with the leaf $L_p$.
By replacing the counting measure on $\R^x \equiv \R[x]$ by the 
Riemannian volume on $\R_M^p \equiv L_p$, we can obtain a {\em Haar system} in the sense of \cite{renault} (also named  {\em invariant transverse function} in \cite{cofw}), that is a family  of measures $\lambda^p$ supported by the fibres 
$\R_M^p$ such that $\lambda^q = \lambda^p$ for all $q \in L_p$. 
According to \cite{cofw}, the amenability of $\R_M$ with respect to $(\lambda,\mu)$ means the existence for $\mu$-almost every $p \in M$ of a mean on $\R_M^p$ ({\it  i.e.}  a  linear map
$m_p : L^\infty(\R_M^p,\lambda^p)\to \rr$ such that $m_p(f)\geq 0$ for $f\geq 0$ and $m_p(1)=1$) so that the function $p \mapsto m_p$ on $\R_M$ is
measurable ({\it  i.e.}  for a measurable function $\tilde{f}$ defined on the 
graph of $\R_M$, the function defined by $f(p)=m_p(\tilde{f}(p,\cdot))$ is measurable) and invariant ({\it  i.e.}  $m_p =m_q$ for all $q \in L_p$). In this context, Reiter's criterion exhibited in the 
theorem above may be formulated as follow: \emph{there exist sequences of probability measures $\{\pi_n^p\}_{p \in M, n\in \nn}$ on $\R_M^p \equiv L_p$, which are absolutely continuous with respect to 
$\lambda^p$, such that 
$\|\pi_n^p-\pi_n^q\|\to 0$ 
for $\widetilde{\mu}$-almost all $(p,q)\in \R_M$. 
The map 
$p \mapsto \pi_n^p$ is measurable for all $n$, in the same sense as before, and $\widetilde{\mu} = \int \lambda^p d\mu(p)$.} For other equivalent conditions, we refer to proposition 3.2.14 of \cite{renault}.

\begin{definition}
We will say that the $\mathcal{F}$ is {\em amenable} if it satisfies one of the following two equivalent condition (see corollary 16 of \cite{cofw}): 
\smallskip 

\noindent
--  the equivalence relation $\R_M$ on $M$ is amenable (with respect to $(\lambda,\mu)$); 
\smallskip 

\noindent
-- the induced induced equivalence relation $\R$ on some (equivalently any) total transversal $T$ is amenable (with respect to the transverse measure $\nu$ induced by $\mu$ on $T$, 
which is well defined up to measure equivalence). 
\end{definition}

\section[Non-amenable F{\o}lner foliations]{Two examples of F{\o}lner foliations that are non-amenable} 
\label{contraejemplo}

In this section we will consider foliations $\F$ with transverse invariant measures $\nu$.
Before the constructions, we are going to explain the reason why the foliations will have dimension and codimension equal to $2$. 
\medskip 

 C. Series' theorem states that if $\F$ has polynomial 
growth, then it is hyperfinite (we refer to \cite{seri}). This result was also proved by M. Samu\'elid\`es \cite{samu} in the quasi-invariant case.
In terms of foliations this result can be stated as:

\begin{theorem}[Series, Samu\'elid\`es]\label{teoremaseries}
A foliation such that $\nu$-almost all its leaves have 
polynomial growth is amenable.
\end{theorem}

\noindent
Kaimanovich gave a proof, valid for foliations with sub-exponential growth \cite{kaim}. 
Combining the latter result with the fact that the leaves in the support of an invariant measure of a codimension one foliation have polynomial growth, we refer to J. F. Plante's theorem (theorem 6.3 of \cite{plan}), we deduce that such a foliation is always amenable. 
\medskip 

As we said, we will give examples of non-amenable foliations of a compact manifold that are F{\o}lner. 
In 2001, Kaimanovich had already constructed some examples of non-amenable discrete graphed equivalence relations with F{\o}lner classes, and an example of a non-amenable F{\o}lner foliation. This foliated example has a Reeb component, and thus the invariant measure is not locally finite. We refer the reader to theorem 3 of \cite{kaim}. The examples we are going to construct have finite transverse invariant measures, and they may be interpreted as foliated versions of examples 1 and 3 described in \cite{kaim}.
\medskip 

The main idea in the following two examples is to construct a closed $4$-manifold $M$ with a non-amenable ergodic foliation $\mathcal{F}_1$ by suspension of the action of a non-amenable group, and then make a surgery so that the leaves become F{\o}lner.

\subsection{First example: inserting a Reeb component}
\label{seccionejemploreeb}

We will begin by constructing an ergodic non-amenable foliation $\mathcal{F}_1$ of a closed $4$-manifold, admitting a total transversal diffeomorphic to a torus. We will then modify the foliation in a neighborhood of the transversal, inserting arbitrarily long 
{\it bumps}, in order to prove the following  proposition:

\begin{proposition} \label{proposicionejemplofernando}
There exists an ergodic measured foliation $\mathcal{F}$ of a closed $4$-manifold possessing a finite transverse invariant measure $\nu$, that is non-amenable and is F{\o}lner.
\end{proposition}

\begin{proof}
Let $\Sigma_2$ be the orientable compact surface of genus two, whose fundamental group $\pi_1(\Sigma_2)$ is generated by the loops $\alpha_1$, $\beta_1$, $\alpha_2$, $\beta_2$, satisfying the relation $[\alpha_1,\beta_1]=[\alpha_2,\beta_2]$.
The quotient of $\pi_1(\Sigma_2)$ by the normal subgroup generated by $\beta_1$ and $\beta_2$
is a free group. Thus, for any free subgroup $\Gamma$ of  $SL(2,\zz)$ with two generators, we have a surjective homomorphism $\phi : \pi_1(\Sigma_2)\to \Gamma$. For example, consider as $\Gamma$ the subgroup generated  by 
$A_1=\matrice{1}{2}{0}{1}$ and $A_2 = \matrice{1}{0}{2}{1}$. In general, we will denote $A_i=\phi(\alpha_i)$, for $i=1,2$.
\medskip 

Let us construct a foliation  $\mathcal{F}_1$ of $M=\Sigma_2\times \mathbb{T}^2$ defined by the suspension of $\phi$. First, we cut $\Sigma_2$ along the loops 
$\beta_1$ and $\beta_2$, to obtain a surface $S$ of genus zero with four boundary components: $\beta_1^\pm$ and $\beta_2^\pm$. Later we consider the manifold $S\times \mathbb{T}^2$ with the product foliation by {\it double pants} $S\times\{\ast\}$, and identify $\beta_i^-\times  \mathbb{T}^2$ with $\beta_i^+\times \mathbb{T}^2$
using $Id \times A_i$, for $i=1,2$. We get $(M,\mathcal{F}_1)$ such that:
\smallskip 

\noindent
-- it has a transverse invariant volume, given by the canonical volume form on $\mathbb{T}^2$, with respect to whom $\mathcal{F}_1$ has no essential holonomy and is ergodic;
\smallskip 

\noindent
--  it is non-amenable, with respect to this volume, since the transverse structure is given by the action of $\Gamma$;

\noindent
-- the leaves without holonomy are diffeomorphic to the {\it Cantor tree}, {\it i.e.} a genus zero surface with a Cantor set of ends.
\smallskip 

Let $T$ be a total transversal diffeomorphic to $\mathbb{T}^2$. The next step in the construction is to change the foliation in the interior of a tubular neighborhood $P=\dd^2\times T$ of the transversal, using a foliation $\mathcal{H}$ of $P$. 
Observe that the leaves of $\mathcal{F}_1$ intersect infinitely many times $P$. In order to prove the proposition we need that $\mathcal{H}$ has an invariant transverse measure and that the insertion of it inside $P$ makes the leaves F{\o}lner. As we said, the idea is to replace discs in the leaves with arbitrarily long bumps.
\medskip 

\noindent
 {\em -- Construction of the foliation $\mathcal{H}$ of $P$}

Let us begin with a classical construction of G. Reeb \cite{reeb}. Consider the solid torus $\dd^2\times \sss^1$ with coordinates $(r,\psi,\theta)$, and let $\mathcal{H}_0$ be the product foliation by discs defined by the equation $d\theta =0$. Let $f_t:[-1,1]\to [0,\infty)$ be a family of $C^\infty$
{\em turbulization functions} with $t\in[0,1]$, such that
\smallskip 

\noindent
--  $f_0$ is identically zero and $f_t$ is symmetric with respect to zero for every $t > 0$; 
\smallskip 

\noindent
--  $f_t$ is decreasing in $[0,\frac{3}{4}]$ and zero in $[\frac{3}{4},1]$ for every $t>0$;
\smallskip 

\noindent
-- $\lim_{t\to 1}f_t(0)=+ \infty$ and $\lim_{t\to 1}\frac{\partial f_t}{\partial r}(\frac{1}{2})=-\infty$.
\smallskip 

\setlength{\unitlength}{1cm}
\begin{figure}
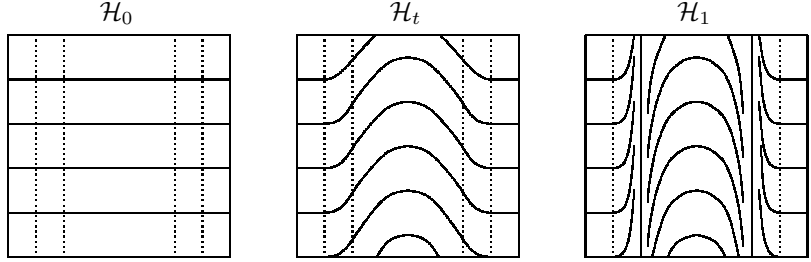

\[
\xy 0;/r.07pc/: 
(0,0)*{}="A1"; (100,0)*{}="F1";
(0,20)*{}="A2";(100,20)*{}="F2"; 
(0,40)*{}="A3"; (100,40)*{}="F3";
(0,60)*{}="A4"; (100,60)*{}="F4";
(0,80)*{}="A5"; (100,80)*{}="F5";
(0,100)*{}="A6"; (100,100)*{}="F6";
(49,110)*{\mathcal{H}_0}="H0";
"A1"; "F1" **\dir{-}; 
"A2"; "F2" **\dir{-}; 
"A3"; "F3" **\dir{-}; 
"A4"; "F4" **\dir{-}; 
"A5"; "F5" **\dir{-}; 
"A6"; "F6" **\dir{-}; 
"A1"; "A6" **\dir{-}; 
"F1"; "F6" **\dir{-}; 
(12.5,0)*{}="B1"; (25,0)*{}="C1";(75,0)*{}="D1";(87.5,0)*{}="E1";
(12.5,100)*{}="B6"; (25,100)*{}="C6";(75,100)*{}="D6";(87.5,100)*{}="E6"; 
"B1"; "B6"**\dir{.}; 
"C1"; "C6"**\dir{.}; 
"D1"; "D6"**\dir{.}; 
"E1"; "E6"**\dir{.}; 
(130,0)*{}="A11"; (230,0)*{}="F11";
(130,20)*{}="A21";(230,20)*{}="F21"; 
(130,40)*{}="A31"; (230,40)*{}="F31";
(130,60)*{}="A41"; (230,60)*{}="F41";
(130,80)*{}="A51"; (230,80)*{}="F51";
(130,100)*{}="A61"; (230,100)*{}="F61";
(142.5,0)*{}="B11"; (217.5,0)*{}="E11";
(142.5,20)*{}="B21"; (217.5,20)*{}="E21"; 
(142.5,40)*{}="B31"; (217.5,40)*{}="E31";
(142.5,60)*{}="B41"; (217.5,60)*{}="E41";
(142.5,80)*{}="B51"; (217.5,80)*{}="E51";
(142.5,100)*{}="B61"; (217.5,100)*{}="E61";
(155,0)*{}="C11"; (205,0)*{}="D11";
(155,100)*{}="C61"; (205,100)*{}="D61";
(179,110)*{\mathcal{H}_t}="Ht";
"A11"; "F11" **\dir{-}; 
"A61"; "F61" **\dir{-}; 
"A21"; "B21" **\dir{-}; 
"B21";(180,50) **\crv{(150,20.1) & (155,30) & (165,42) & (170,47) & (175,49.5) & (180,50)}; 
(180,50);"E21" **\crv{(180,50) & (185,49.5) & (190,47)  & (195,42) & (205,30) & (210,20.1)}; 
"A31"; "B31" **\dir{-}; 
"B31";(180,70) **\crv{(150,40.1) & (155,50) & (165,62) & (170,67) & (175,69.5) & (180,70)}; 
(180,70);"E31" **\crv{(180,70) & (185,69.5) & (190,67)  & (195,62) & (205,50) & (210,40.1)}; 
"A41"; "B41" **\dir{-}; 
"B41";(180,90) **\crv{(150,60.1) & (155,70) & (165,82) & (170,87) & (175,89.5) & (180,90)}; 
(180,90);"E41" **\crv{(180,90) & (185,89.5) & (190,87)  & (195,82) & (205,70) & (210,60.1)}; 
"A51"; "B51" **\dir{-}; 
"B51";(165.5,100) **\crv{(150,80.1) & (155,90)}; 
(194.5,100);"E51" **\crv{(205,90) & (210,80.1)}; 
(166,0);(180,10) **\crv{(170,7) & (175,9.5) & (180,10)}; 
(180,10);(194,0) **\crv{(180,10) & (185,9.5) & (190,7)}; 
"B11";(180,30) **\crv{(150,0.1) & (155,8) & (165,22) & (170,27) & (175,29.5) & (180,30)}; 
(180,30);"E11" **\crv{(180,30) & (185,29.5) & (190,27)  & (195,22) & (205,8) & (210,0.1)}; 
"E21"; "F21" **\dir{-}; 
"E31"; "F31" **\dir{-}; 
"E41"; "F41" **\dir{-}; 
"E51"; "F51" **\dir{-}; 
"A11"; "A61" **\dir{-}; 
"F11"; "F61" **\dir{-}; 
"B11"; "B61"**\dir{.}; 
"C11"; "C61"**\dir{.}; 
"D11"; "D61"**\dir{.}; 
"E11"; "E61"**\dir{.}; 
(260,0)*{}="A12"; (360,0)*{}="F12";
(260,20)*{}="A22"; (360,20)*{}="F22"; 
(260,40)*{}="A32"; (360,40)*{}="F32";
(260,60)*{}="A42"; (360,60)*{}="F42";
(260,80)*{}="A52"; (360,80)*{}="F52";
(260,100)*{}="A62"; (360,100)*{}="F62";
(272.5,0)*{}="B12"; (347.5,0)*{}="E12";
(272.5,20)*{}="B22"; (347.5,20)*{}="E22"; 
(272.5,40)*{}="B32"; (347.5,40)*{}="E32";
(272.5,60)*{}="B42"; (347.5,60)*{}="E42";
(272.5,80)*{}="B52"; (347.5,80)*{}="E52";
(272.5,100)*{}="B62"; (347.5,100)*{}="E62";
(285,0)*{}="C12"; (335,0)*{}="D12";
(285,100)*{}="C62"; (335,100)*{}="D62";
(309,110)*{\mathcal{H}_1}="H1";
"A12"; "F12" **\dir{-}; 
"A62"; "F62" **\dir{-}; 
"A22"; "B22" **\dir{-}; 
(290,0);(310,30) **\crv{(290,0) & (295,22) & (302,27) & (305,29.5) & (310,30)}; 
(310,30);(329,0) **\crv{(310,30) & (315,29.5) & (318,27)  & (325,22) & (329,0)}; 
"B22";(282,50) **\crv{(280,20.1) & (280,30)}; 
(339,50);"E22" **\crv{(340,30) & (340,20.1)}; 
(288,5);(310,50) **\crv{(289,20) & (295,42) & (302,47) & (305,49.5) & (310,50)}; 
(310,50);(331,5) **\crv{(310,50) & (315,49.5) & (318,47)  & (325,42) & (330,20)}; 
"A32"; "B32" **\dir{-}; 
"B32";(282,70) **\crv{(280,40.1) & (280,50)}; 
(338,70);"E32" **\crv{(340,50) & (340,40.1)}; 
(288,25);(310,70) **\crv{(289,40) & (295,62) & (302,67) & (305,69.5) & (310,70)}; 
(310,70);(331,25) **\crv{(310,70) & (315,69.5) & (318,67)  & (325,62) & (330,40)}; 
"A42"; "B42" **\dir{-}; 
"B42";(282,90) **\crv{(280,60.1) & (280,70)}; 
(338,90);"E42" **\crv{(340,70) & (340,60.1)}; 
(288,45);(310,90) **\crv{(289,60) & (295,82) & (302,87) & (305,89.5) & (310,90)}; 
(310,90);(331,45) **\crv{(310,90) & (315,89.5) & (318,87)  & (325,82) & (330,60)}; 
"A52"; "B52" **\dir{-}; 
"B52";(282,100) **\crv{(280,80.1) & (280,90)}; 
(338,100);"E52" **\crv{(340,90) & (340,80.1)}; 
"B12";(282,30) **\crv{(280,0.1) & (280,10)}; 
(338,30);"E12" **\crv{(340,10) & (340,0.1)}; 
(288,65);(296,100) **\crv{(289,80) & (291.5,90) & (296,100)}; 
(324,100);(331,65) **\crv{(324,100) & (328,90) & (330,80)}; 
(297,0);(310,10) **\crv{(300,7) & (307,9.5) & (310,10)}; 
(310,10);(323,0) **\crv{(310,10) & (313,9.5) & (320,7)}; 
"E22"; "F22" **\dir{-}; 
"E32"; "F32" **\dir{-}; 
"E42"; "F42" **\dir{-}; 
"E52"; "F52" **\dir{-}; 
"A12"; "A62" **\dir{-}; 
"F12"; "F62" **\dir{-}; 
"B12"; "B62"**\dir{.}; 
"C12"; "C62"**\dir{-}; 
"D12"; "D62"**\dir{-}; 
"E12"; "E62"**\dir{.}; 
\endxy
\]
 \caption{\label{fig:foliacionH} The foliations $\mathcal{H}_0$, $\mathcal{H}_t$ for $0<t<1$ and $\mathcal{H}_1$.}
\end{figure}

Using the vector field $X_t=\frac{\partial f_t}{\partial t}\frac{\partial}{\partial\theta}$, we can push the discs of $\mathcal{H}_0$ in the $\sss^1$-direction to get a foliation $\mathcal{H}_t$ isotopic to $\mathcal{H}_0$.  In fact, the isotopy is given by $\Phi_t(r,\psi,\theta)=(r, \psi, \theta + f_t(r))$. The foliations $\mathcal{H}_t$ define a homotopy in the space of foliations from $\mathcal{H}_0$ and $\mathcal{H}_1$. We will call {\it Reeb foliation} the $C^\infty$ foliation $\mathcal{H}$ of $\dd^2\times\sss^1\times [0,1]$ defined by this
homotopy, see figure \ref{fig:foliacionH} where $\dd^2$ is horizontal and $\sss^1$ vertical. The leaves of $\mathcal{H}$ meet $\partial\dd^2\times \sss^1\times [0,1]$ along the circles $\partial\dd^2\times\{\theta\}\times\{t\}$.
Observe that $\mathcal{H}_1$ has a Reeb component inside the solid torus 
$\{ (r,\psi,\theta) \in  \dd^2\times \sss^1 | r\leq \frac{1}{2}\}$. In the complement of this component, $\mathcal{H}$ has a transverse invariant volume given by the differential form 
$\omega= (\Phi_t)_* d\theta\wedge dt$, but the existence of a Reeb component prevents the foliation to have a transverse invariant volume. To find a transverse invariant measure, let us consider the volume form $\omega_0=d\theta\wedge dt$ on the annulus $A_0=\{x_0\}\times\sss^1\times[0,1]$, for a point $x_0\in \partial \dd^2$. The next lemma guarantees the existence of a transverse invariant measure under a convenient choice of the functions $f_t$.

\begin{lemma}
\label{lemaejemplo1}
There exists a Reeb foliation $\mathcal{H}$ of $\dd^2\times \sss^1\times [0,1]$ with a finite transverse invariant measure.
\end{lemma}

\begin{proof}
We want to find the necessary conditions so that $\mathcal{H}$ possesses a transverse invariant measure. First, let us construct a total transversal to $\mathcal{H}$. For the point $x_0\in\partial \dd^2$, take the arc 
$C=\{(rx_0,\theta)\in \dd^2\times \sss^1|\,a<r<b, \,\, \theta=g(r)\}$,
where $a\in (0,\frac{1}{2})$, $b\in(\frac{1}{2},1)$ and $g:(a,b)\to \rr$ is a increasing $C^\infty$ function. Clearly, the union $A=A_0\cup (C\times[0,1])$ is a total transversal. The transverse invariant volume form $\omega$, defined on the complement of the Reeb component, gives us a transverse invariant measure on $A_0\cup(C\times [0,1))$. We want to extend this measure to a transverse invariant measure $\nu$ of $\mathcal{H}$. This is possible if the following integral is finite
\begin{eqnarray*}
\int_{C\times[0,1]}\omega \quad  =  \quad  \int_0^1 ( \int_C  ( \Phi_t)_* d\theta ) dt & = &
\int_0^1 ( \int_a^b ( \frac{\partial g}{\partial r}+\frac{\partial f_t}{\partial r} )dr )dt  \\ 
& = &
 g(b)-g(a)+\int_0^1(f_t(a)-f_t(b))dt.
\end{eqnarray*}
Observe that for a good choice of $f_t$ the function $f_t(a)-f_t(b)$ is integrable, and thus the integral is finite. We can define  the measure $\nu$ as follows: for any Borel set $B\subset C\times[0,1]$, define $\nu(B)=\int_{B\cap(C\times[0,1))}\omega$.  Hence, the transverse measure $\nu$ is defined on $A$ and it is invariant under the holonomy. Notice that the restriction of $\nu$ to $A_0$ is $\omega_0$.
\end{proof}

\setlength{\unitlength}{1cm}
\begin{figure}
\begin{picture}(5,5.5)
\put(0,-0.5){\includegraphics[width=50mm]{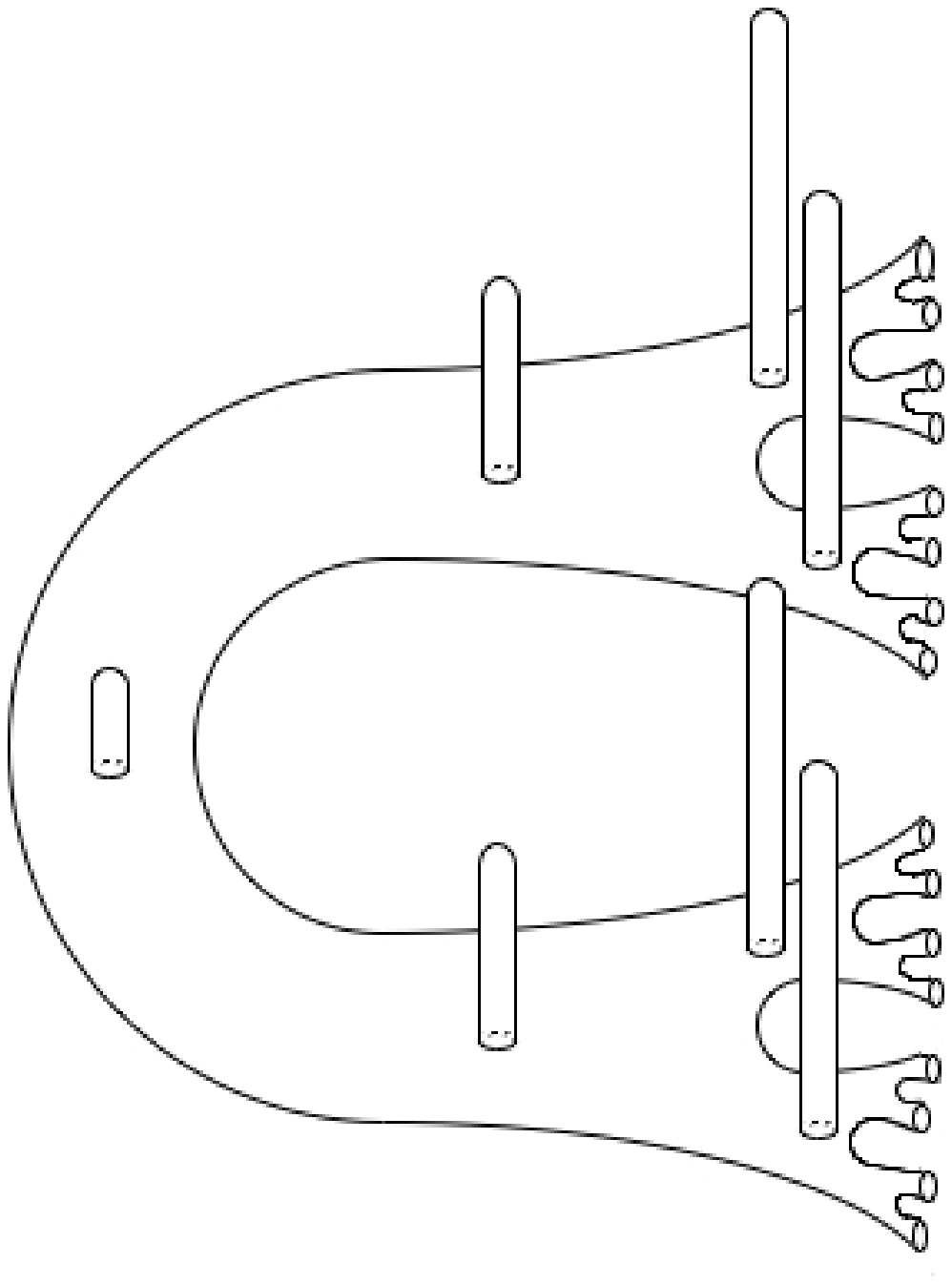}}
\end{picture}
\caption{\label{fig:arbolescantor} A Cantor tree with bumps.}
\end{figure}

We can take two copies of this foliated manifold and paste them in order to get a {\it doubled foliation} on $\dd^2\times \mathbb{T}^2$.  We will abuse of the notation by keeping calling the foliation $\mathcal{H}$. This foliation is endowed with a finite transverse invariant measure $\nu$. 
\medskip

\noindent
{\em -- Insertion of $\mathcal{H}$}

Consider the foliated manifold $(M,\mathcal{F}_1)$. Consider the transverse torus $T$ and a tubular neighborhood $P \cong \dd^2\times T$. Since $\mathcal{H}$ is transverse to the boundary of $\dd^2\times \mathbb{T}^2$, we can replace the foliation in the interior of $P$ by $\mathcal{H}$. Observe that every non compact leaf intersects the tubular neighborhood $P$ in infinitely many discs approaching any of its ends. In such a leaf, the discs are replaced either by 
\smallskip 

\noindent
--  bumps that get longer and longer as we approach the ends, and whose boundary is constant an equal to the boundary of the disc we removed;
\smallskip 

\noindent
-- by semi-infinite cylinders $\sss^1\times [0,\infty)$ that accumulate on the Reeb component.
\smallskip 

\noindent
In both cases, we get F{\o}lner sequences that approach any end. On the other hand, $\mathcal{F}$ has an ergodic transverse invariant measure (obtained from $\nu$) and is non-amenable since $\mathcal{F}$ and $\mathcal{F}_1$ have the same transverse structure in restriction to the closed set $\overline{M \setminus P}$ having $T$ as total transversal. Moreover, the set of leaves that are quasi-isometric to the Cantor tree with bumps, as in figure \ref{fig:arbolescantor}, has total mass. This implies that the foliation is F{\o}lner. 
\end{proof}

\subsection{Second example: using Wilson's plug}
\label{seccionejemplowilson}

We will start by constructing a new non-amenable foliation $\mathcal{F}_1$ of a compact manifold by suspension and then we will perform a modification to make its leaves F{\o}lner. For the modification we will use a volume preserving version of F. W. Wilson's plug \cite{wils}. Let us define plugs for $1$-dimensional foliations, the generalization to bigger dimensions is straightforward.

\begin{definition}
\label{deftrampas} 
A {\em plug} is a manifold $P$ endowed with a $1$-dimensional foliation $\mathcal{H}$ satisfying the following characteristics: the 3-manifold $P$ is of the form $D\times [-1,1]$, where $D$ is an oriented compact 2-manifold with boundary, and 
$\mathcal{H}$ is defined by the orbits of a vector field $H$ such that:
\smallskip 

\noindent
(i)  $\mathcal{H}$ is vertical  in a neighborhood of $\partial P$, that is 
$H = \frac{\partial}{\partial z}$ where $z$ is the coordinate of the interval $[-1,1]$;
\smallskip 

\noindent
(ii)  $\mathcal{H}$ is transverse to $D \times\{\pm 1\}$, we will call $D \times\{-1\}$ the entry region and $D\times\{+1\}$ the exit region;
\smallskip 

\noindent
(iii)  the entry-exit condition is satisfied: if two points $(x,-1)$ and $(y,1)$ are in the same leaf, then $x=y$.
That is a leaf that traverses $P$, exits just in front of its entry point;
\smallskip 

\noindent
(iv)  there is at least one entry point whose  leaf has one end
 contained in $P$, we will say that the leaf is trapped by $P$.
\end{definition}

If a manifold $P$ as above endowed with a foliation satisfies all conditions except {\it (iii)}, we will call it a {\em semi-plug}. The {\it mirror-image} construction allows us to get a plug. Up to rescaling, such a plug consist of the quotient of the semi-plug $(P, \mathcal{H})$ and its mirror-image $(P,-\mathcal{H})$ on $D\times [-1,1]$ by the equivalence relation which identifies their exit and entry regions 
$D\times\{\pm 1\}$.

\setlength{\unitlength}{1cm}
\begin{figure}
\subfigure[Flow lines of $Z$]{
\includegraphics[width=55mm]{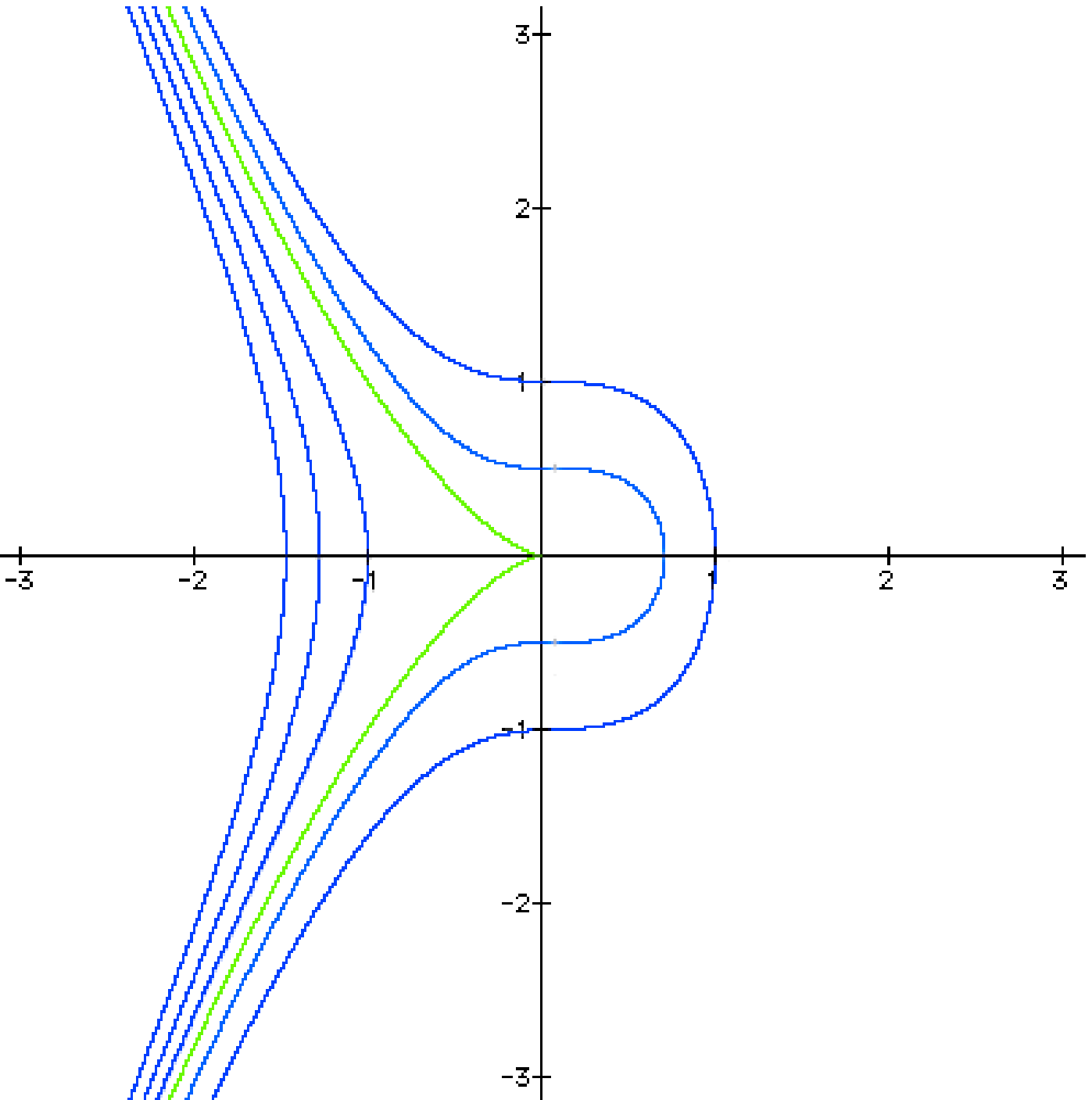} 
\label{fig:hamilton}
}
\hspace{1cm}
\subfigure[Flow lines of $H_1$]{
 \begin{picture}(5,6.3)
 \put(-0.2,0){\includegraphics[width=55mm]{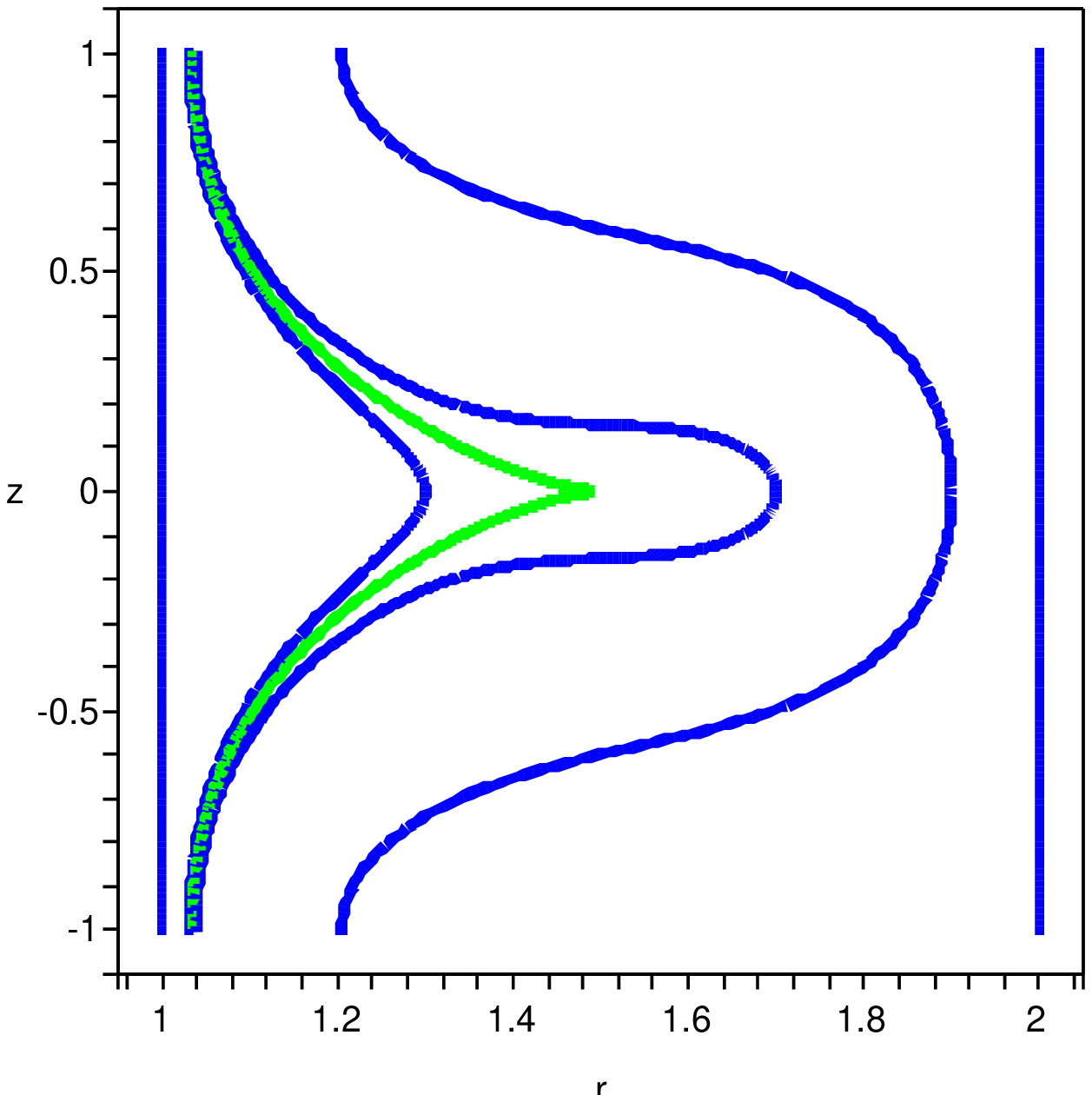}}
 \put(2.5,2.75){\makebox(0,0){$\bullet$}}
 \end{picture}
 \label{fig:curvasdeH1}
 } 
\caption{Flow lines of $Z$ and $H_1$.}
\end{figure}

\noindent
{\em --  A volume preserving version of Wilson's plug}

The original construction of Wilson \cite{wils}, gives a $3$-dimensional plug with a vector field that cannot be volume  preserving. Here we need a volume preserving analogue, 
 and we will use a method of construction based on an idea of Ghys. Consider the plane 
$\rr^2$ with coordinates $(r,z)$ and the Hamiltonian vector field 
$Z = -2z \frac{\partial}{\partial r} + 3r^2 \frac{\partial}{\partial z}$
associated to the energy function $\zeta(r,z) = r^3 +  z^2$, see figure \ref{fig:hamilton}.
By definition, the flow $\phi_t$ fixes the origin, preserves the energy levels 
$\{ \zeta=const.\}$ and the symplectic form $\omega = dr \wedge dz$. 

We are going to modify this example to make it fit  inside the rectangle $R = [1,2] \times [-1,1]$.  An area preserving vector field with the conditions we need is given by the Hamiltonian vector field $H_1$ on $R$ associated to the function
$$
h(r,z)=(r-\frac{3}{2})^3+(z^2-\frac{1}{2}z^4)g(r),
$$
where $g$ is a $C^\infty$ function equal to zero near $1$ and $2$ and such that the function
$3(r-\frac{3}{2})^2+(z^2-\frac{1}{2}z^4)g^\prime(r)$ is positive away of the singularity $(\frac{3}{2},0)$,  see figure \ref{fig:curvasdeH1}. Take on the manifold $P = \sss^1 \times R$ the vector field 
$ \widehat{H} =H_1 + f\frac{\partial}{\partial \theta}$, where 
$f:R\to \rr^+$ is a $C^\infty$ function that assumes the value zero near the boundary of the rectangle and is strictly positive on the singularity. Thus, $\widehat{H}$ is non singular and has a periodic orbit. Moreover, it preserves the volume form $d\theta\wedge \omega$. In order to satisfy the entry-exit condition we are going to use the mirror-image construction. Rescaling the vector fields to make the two copies fit inside $P$, we will call the new plug $P$ and the resulting vector field $H$. Observe that $H$ is non singular, has two periodic orbits and preserves the volume form $d\theta\wedge\omega$.
\medskip 

Hence, we have endowed the manifold $P$ with a $1$-dimensional foliation $\mathcal{H}$. Observe that there is a $2$-dimensional foliation of $P$ that is tangent to $H$: it is defined by the equation $h=const.$ and the leaves are the product of $\sss^1$ with the two copies of the flow lines of $H_1$. It is convenient to use the system of coordinates $(\theta, h, z)$.  Summarizing, we have that:
\smallskip 

\noindent
--  $\mathcal{H}$ has a transverse invariant volume, given by the differential $2$-form $\iota_H(d\theta \wedge \omega)$;
\smallskip 

\noindent
-- $\mathcal{H}$ satisfies the entry-exit condition;
\smallskip 

\noindent
-- $\mathcal{H}$ has trapped leaves: the leaves through the points in $(\theta, 0,z)$  are  semi-infinite or infinite, except for the two circle leaves corresponding to the two copies of the  singularity of $H_1$. The flow lines of $H$  in the cylinders $\{h=const.\}$  are similar to described in figure 
\ref{fig:measuredwilson} for $h = \pm\frac{1}{8}$, $h\approx 0$ and $h=0$.

\setlength{\unitlength}{1cm}
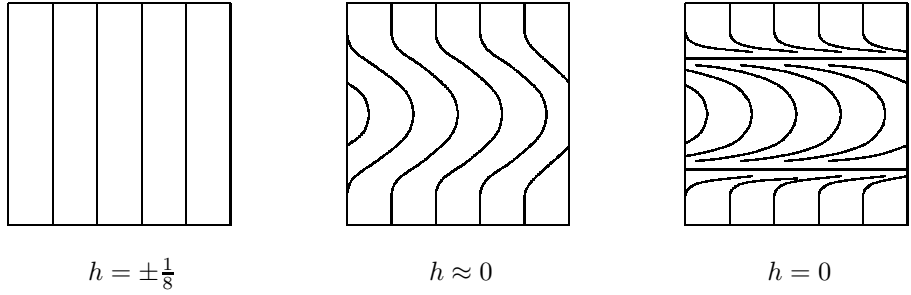
\begin{figure}
\begin{picture}(12,4)
\put(0.85,0){ $h = \pm \frac 1 8$}
\put(-0.3,3.8){
\rotatebox{-90}{
$\xy 0;/r.07pc/: 
(0,0)*{}="A1"; (100,0)*{}="F1";
(0,20)*{}="A2";(100,20)*{}="F2"; 
(0,40)*{}="A3"; (100,40)*{}="F3";
(0,60)*{}="A4"; (100,60)*{}="F4";
(0,80)*{}="A5"; (100,80)*{}="F5";
(0,100)*{}="A6"; (100,100)*{}="F6";
"A1"; "F1" **\dir{-}; 
"A2"; "F2" **\dir{-}; 
"A3"; "F3" **\dir{-}; 
"A4"; "F4" **\dir{-}; 
"A5"; "F5" **\dir{-}; 
"A6"; "F6" **\dir{-}; 
"A1"; "A6" **\dir{-}; 
"F1"; "F6" **\dir{-}; 
(12.5,0)*{}="B1"; (25,0)*{}="C1";(75,0)*{}="D1";(87.5,0)*{}="E1";
(12.5,100)*{}="B6"; (25,100)*{}="C6";(75,100)*{}="D6";(87.5,100)*{}="E6"; 
\endxy$}}
\put(5.5,0){$h \approx 0$}
\put(4.2,3.8){
\rotatebox{-90}{
$\xy 0;/r.07pc/: 
(130,0)*{}="A11"; (230,0)*{}="F11";
(130,20)*{}="A21";(230,20)*{}="F21"; 
(130,40)*{}="A31"; (230,40)*{}="F31";
(130,60)*{}="A41"; (230,60)*{}="F41";
(130,80)*{}="A51"; (230,80)*{}="F51";
(130,100)*{}="A61"; (230,100)*{}="F61";
(142.5,0)*{}="B11"; (217.5,0)*{}="E11";
(142.5,20)*{}="B21"; (217.5,20)*{}="E21"; 
(142.5,40)*{}="B31"; (217.5,40)*{}="E31";
(142.5,60)*{}="B41"; (217.5,60)*{}="E41";
(142.5,80)*{}="B51"; (217.5,80)*{}="E51";
(142.5,100)*{}="B61"; (217.5,100)*{}="E61";
(155,0)*{}="C11"; (205,0)*{}="D11";
(155,100)*{}="C61"; (205,100)*{}="D61";
"A11"; "F11" **\dir{-}; 
"A61"; "F61" **\dir{-}; 
"A21"; "B21" **\dir{-}; 
"B21";(180,50) **\crv{(150,20.1) & (155,30) & (165,42) & (170,47) & (175,49.5) & (180,50)}; 
(180,50);"E21" **\crv{(180,50) & (185,49.5) & (190,47)  & (195,42) & (205,30) & (210,20.1)}; 
"A31"; "B31" **\dir{-}; 
"B31";(180,70) **\crv{(150,40.1) & (155,50) & (165,62) & (170,67) & (175,69.5) & (180,70)}; 
(180,70);"E31" **\crv{(180,70) & (185,69.5) & (190,67)  & (195,62) & (205,50) & (210,40.1)}; 
"A41"; "B41" **\dir{-}; 
"B41";(180,90) **\crv{(150,60.1) & (155,70) & (165,82) & (170,87) & (175,89.5) & (180,90)}; 
(180,90);"E41" **\crv{(180,90) & (185,89.5) & (190,87)  & (195,82) & (205,70) & (210,60.1)}; 
"A51"; "B51" **\dir{-}; 
"B51";(165.5,100) **\crv{(150,80.1) & (155,90)}; 
(194.5,100);"E51" **\crv{(205,90) & (210,80.1)}; 
(166,0);(180,10) **\crv{(170,7) & (175,9.5) & (180,10)}; 
(180,10);(194,0) **\crv{(180,10) & (185,9.5) & (190,7)}; 
"B11";(180,30) **\crv{(150,0.1) & (155,8) & (165,22) & (170,27) & (175,29.5) & (180,30)}; 
(180,30);"E11" **\crv{(180,30) & (185,29.5) & (190,27)  & (195,22) & (205,8) & (210,0.1)}; 
"E21"; "F21" **\dir{-}; 
"E31"; "F31" **\dir{-}; 
"E41"; "F41" **\dir{-}; 
"E51"; "F51" **\dir{-}; 
"A11"; "A61" **\dir{-}; 
"F11"; "F61" **\dir{-}; 
\endxy$}}
\put(10,0){$h = 0$}
\put(8.7,3.8){
\rotatebox{-90}{
$\xy 0;/r.07pc/: 
(260,0)*{}="A12"; (360,0)*{}="F12";
(260,20)*{}="A22"; (360,20)*{}="F22"; 
(260,40)*{}="A32"; (360,40)*{}="F32";
(260,60)*{}="A42"; (360,60)*{}="F42";
(260,80)*{}="A52"; (360,80)*{}="F52";
(260,100)*{}="A62"; (360,100)*{}="F62";
(272.5,0)*{}="B12"; (347.5,0)*{}="E12";
(272.5,20)*{}="B22"; (347.5,20)*{}="E22"; 
(272.5,40)*{}="B32"; (347.5,40)*{}="E32";
(272.5,60)*{}="B42"; (347.5,60)*{}="E42";
(272.5,80)*{}="B52"; (347.5,80)*{}="E52";
(272.5,100)*{}="B62"; (347.5,100)*{}="E62";
(285,0)*{}="C12"; (335,0)*{}="D12";
(285,100)*{}="C62"; (335,100)*{}="D62";
"A12"; "F12" **\dir{-}; 
"A62"; "F62" **\dir{-}; 
"A22"; "B22" **\dir{-}; 
(290,0);(310,30) **\crv{(290,0) & (295,22) & (302,27) & (305,29.5) & (310,30)}; 
(310,30);(329,0) **\crv{(310,30) & (315,29.5) & (318,27)  & (325,22) & (329,0)}; 
"B22";(282,50) **\crv{(280,20.1) & (280,30)}; 
(339,50);"E22" **\crv{(340,30) & (340,20.1)}; 
(288,5);(310,50) **\crv{(289,20) & (295,42) & (302,47) & (305,49.5) & (310,50)}; 
(310,50);(331,5) **\crv{(310,50) & (315,49.5) & (318,47)  & (325,42) & (330,20)}; 
"A32"; "B32" **\dir{-}; 
"B32";(282,70) **\crv{(280,40.1) & (280,50)}; 
(338,70);"E32" **\crv{(340,50) & (340,40.1)}; 
(288,25);(310,70) **\crv{(289,40) & (295,62) & (302,67) & (305,69.5) & (310,70)}; 
(310,70);(331,25) **\crv{(310,70) & (315,69.5) & (318,67)  & (325,62) & (330,40)}; 
"A42"; "B42" **\dir{-}; 
"B42";(282,90) **\crv{(280,60.1) & (280,70)}; 
(338,90);"E42" **\crv{(340,70) & (340,60.1)}; 
(288,45);(310,90) **\crv{(289,60) & (295,82) & (302,87) & (305,89.5) & (310,90)}; 
(310,90);(331,45) **\crv{(310,90) & (315,89.5) & (318,87)  & (325,82) & (330,60)}; 
"A52"; "B52" **\dir{-}; 
"B52";(282,100) **\crv{(280,80.1) & (280,90)}; 
(338,100);"E52" **\crv{(340,90) & (340,80.1)}; 
"B12";(282,30) **\crv{(280,0.1) & (280,10)}; 
(338,30);"E12" **\crv{(340,10) & (340,0.1)}; 
(288,65);(296,100) **\crv{(289,80) & (291.5,90) & (296,100)}; 
(324,100);(331,65) **\crv{(324,100) & (328,90) & (330,80)}; 
(297,0);(310,10) **\crv{(300,7) & (307,9.5) & (310,10)}; 
(310,10);(323,0) **\crv{(310,10) & (313,9.5) & (320,7)}; 
"E22"; "F22" **\dir{-}; 
"E32"; "F32" **\dir{-}; 
"E42"; "F42" **\dir{-}; 
"E52"; "F52" **\dir{-}; 
"A12"; "A62" **\dir{-}; 
"F12"; "F62" **\dir{-}; 
"C12"; "C62"**\dir{-}; 
"D12"; "D62"**\dir{-}; 
\endxy$}}
\end{picture}
\caption{\label{fig:measuredwilson}  Flow lines of $H$ in the cylinders $\{h=const.\}$.}
\end{figure}

\begin{proposition}
\label{proposicionejemplo}
There exists a non-amenable F{\o}lner foliation $\mathcal{F}$ of a closed $4$-manifold $M$ with an ergodic transverse invariant volume.
\end{proposition}
 
\begin{proof}
To construct the non-amenable two dimensional foliation $\mathcal{F}_1$ we will use again the suspension of the action of a non-amenable group, this time acting on the sphere. Let us begin by considering an oriented surface $\Sigma_2$ of genus two. 
Take now a homomorphism 
$\phi:\pi_1(\Sigma_2,x_0)\to SO(3)$ where $\phi(\beta_1)=\phi(\beta_2) = Id$ and $\Gamma=\phi(\pi_1(\Sigma_2,x_0))$ is a dense free subgroup of $SO(3)$.
As in the previous example,  
we obtain a foliation $\mathcal{F}_1$ of  the manifold $M=\Sigma_2 \times\sss^2$ such that
\smallskip 

\noindent
--  it has a transverse invariant volume, given by the canonical volume form $\Omega$ on $\sss^2$, with respect to whom $\mathcal{F}_1$ has no essential holonomy and is ergodic;
\smallskip 

\noindent
--  it is non-amenable with respect to this volume;
\smallskip 

\noindent
--   the leaves without holonomy are diffeomorphic to the Cantor tree. 
\smallskip 

\noindent
To make the leaves F{\o}lner we will insert a $2$-dimensional volume preserving foliated plug. 
\medskip 

\noindent
{\em -- Construction of the foliated plug.}

Fix a volume preserving plug $P=  \sss^1 \times  [1,2]  \times [-1,1]$  endowed with the
$1$-dimension foliation $\mathcal{H}$, as before. There is an embedding of  $P$  into the product $\sss^2 \times [-1,1]$ such that $H$ extends trivially to the whole product $\sss^2 \times [-1,1]$. We will abuse of the notation by keeping calling the new foliation $\mathcal{H}$. Obviously the vector field defining this foliations is volume preserving. By J. Moser's theorem from \cite{mose}, we can assume that the volume form induced on the entry region $\sss^2 \times \{-1\}$ coincides with the canonical volume form $\Omega$. 
\medskip 

Consider the manifold $Q= \sss^2 \times [-1,1] \times \sss^1$ foliated by $\mathcal{G} = \mathcal{H}\times\sss^1$, that possesses a transverse invariant volume given by the 
$2$-form $\iota_H(d\theta\wedge \omega)$. Observe that the entry region $\sss^2 \times\{-1\}  \times \sss^1$ is foliated by circles. The circles $(\theta,h,\pm 1, \beta)$, with $h\neq 0$ and $\beta\in \sss^1$, belong to the same leaf of the foliation $\mathcal{G}$.
\medskip

\noindent
{\em -- Insertion of the plug}

Let us find a place in $M$ to insert the plug. 
Consider the submanifold $T= \{x_0\}\times \sss^2$. Then $T$ is a transversal to $\mathcal{F}_1$ diffeomorphic to $\sss^2$. 
Consider also a tubular neighborhood $A \cong \sss^1 \times [-1,1]$ of the loop  $\beta_2$ passing through $x_0$. Since the holonomy  on $\beta_2$ is trivial, the trace of $ \mathcal{F}_1$ on the product 
$B = A \times T \cong \sss^1 \times [-1,1] \times \sss^2$ is conjugated to the horizontal foliation (whose leaves are the cylinders  $\sss^1 \times [-1,1] \times \{\ast\}$). As in the previous example, after reordering the factors, we can replace the trivial foliated subbundle $B$ by the plug $Q$. Obviously the resulting foliation $\mathcal{F}$ has a transverse invariant volume.
\medskip 

\noindent
{\em -- The foliation $\mathcal{F}$}

After the insertion of the plug $Q$, we have a $C^\infty$ foliation $\mathcal{F}$ of $M$ with a transverse invariant volume. We need to show that $\mathcal{F}$ is non-amenable and that all its leaves are F{\o}lner. For the first, notice that the equivalence relation on any transversal that does not intersects the inserted plug has not changed. Thus $\mathcal{F}$ is non-amenable.
\medskip 

For the second point, let us analyze the leaves of the foliation $\mathcal{F}$. The leaves that entry $Q$ in points with coordinates $(\theta,h,1,\beta)$ for $\theta, \beta \in \sss^1$ and $h\neq 0$ do not change: for $h\approx \pm \frac{1}{8}$ we just remove a cylinder and put the same back in its place. As $h\to 0$ the cylinder we glue back becomes longer, it turns inside $Q$ as the flow lines of $H$ in $P$ (see figure \ref{fig:measuredwilson}). For each leaf of $\mathcal{F}$ that meets points of coordinates $(\theta, 0, \pm 1,\beta)$ we remove a cylinder and glue back a semi-infinite cylinder $[0,\infty)\times \sss^1$ that corresponds to a flow line in the singular cylinder of $P$  spiralling and accumulating on the periodic orbit.
We have also created new leaves: two compact tori  corresponding to the periodic orbits and infinite cylinders that correspond to the flow lines in the singular cylinder of $P$ that lie between the two periodic orbits.
Clearly, all the leaves that meet this singular cylinder become F{\o}lner. Since 
$\mathcal{F}_1$ is minimal, every leaf becomes F{\o}lner: all the generic leaves are quasi-isometric to 
the  Cantor tree with tubes, see a portion in figure \ref{fig:arbolescantortubos}.
\end{proof}

\begin{figure}
\rotatebox{-90}{
\begin{picture}(3.5,10)(0.2,-0.6)
\put(1.6,0){$\xy 
(0,-3)*\ellipse(3,1){.}; 
(0,-3)*\ellipse(3,1)__,=:a(-180){-}; 
(-3,6)*\ellipse(3,1){.}; 
(-3,6)*\ellipse(3,1)__,=:a(-180){-}; 
(3,6)*\ellipse(3,1){-}; 
(-3,12)*{}="1"; 
(3,12)*{}="2"; 
(-9,12)*{}="A2";
(9,12)*{}="B2"; 
"1";"2" **\crv{(-3,7) & (3,7)};
(-3,-6)*{}="A0";
(3,-6)*{}="B0"; 
(-3,1)*{}="A1";
(3,1)*{}="B1"; 
"A0";"A1" **\dir{-}; 
"B0";"B1" **\dir{-}; 
"B2";"B1" **\crv{(8,7) & (3,5)}; 
"A2";"A1" **\crv{(-8,7) & (-3,5)}; 
\endxy$}
\put(1.6,1.3){\line(0,1){1.2}} 
\put(2.2,1.3){\line(0,1){1.2}} 
\put(1,3){$\xy 
(0,-3)*\ellipse(3,1){.}; 
(0,-3)*\ellipse(3,1)__,=:a(-180){-}; 
(-3,6)*\ellipse(3,1){.}; 
(-3,6)*\ellipse(3,1)__,=:a(-180){-}; 
(3,6)*\ellipse(3,1){-}; 
(-3,12)*{}="1"; 
(3,12)*{}="2"; 
(-9,12)*{}="A2";
(9,12)*{}="B2"; 
"1";"2" **\crv{(-3,7) & (3,7)};
(-3,-6)*{}="A0";
(3,-6)*{}="B0"; 
(-3,1)*{}="A1";
(3,1)*{}="B1"; 
"A0";"A1" **\dir{-}; 
"B0";"B1" **\dir{-}; 
"B2";"B1" **\crv{(8,7) & (3,5)}; 
"A2";"A1" **\crv{(-8,7) & (-3,5)}; 
\endxy$}
\put(1,4.3){\line(0,1){3.6}} 
\put(1.6, 4.3){\line(0,1){3.6}} 
\put(0.4,7.9){$\xy 
(0,-3)*\ellipse(3,1){.}; 
(0,-3)*\ellipse(3,1)__,=:a(-180){-}; 
(-3,6)*\ellipse(3,1){-}; 
(3,6)*\ellipse(3,1){-}; 
(-3,12)*{}="1"; 
(3,12)*{}="2"; 
(-9,12)*{}="A2";
(9,12)*{}="B2"; 
"1";"2" **\crv{(-3,7) & (3,7)};
(-3,-6)*{}="A0";
(3,-6)*{}="B0"; 
(-3,1)*{}="A1";
(3,1)*{}="B1"; 
"A0";"A1" **\dir{-}; 
"B0";"B1" **\dir{-}; 
"B2";"B1" **\crv{(8,7) & (3,5)}; 
"A2";"A1" **\crv{(-8,7) & (-3,5)}; 
\endxy$}
\end{picture}}
\caption{\label{fig:arbolescantortubos} A portion of ``Cantor tree'' with tubes.}
\end{figure}
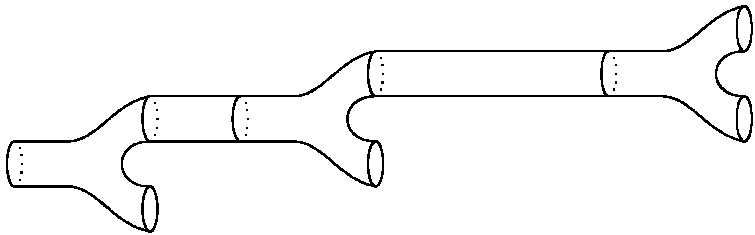

\begin{remarks} 
-- Both examples can be made real analytic. In the first case, we can use U. Hirsch's method from \cite{uhirsch} in order to insert a real analytic Reeb foliation. For the other one, as explained in \cite{ghyskupe} and \cite{2kupe}, plug insertion can be made in the real analytic  category.
\medskip 

\noindent
--  The second example has exactly two types of invariant (harmonic) probability measures: one is diffuse defined by the transverse volume and the other are combinations of atomic measures supported by the toric leaves. On the contrary, the first example has a countable infinite set of atomic invariant measures supported by compact leaves, one toric and the others hyperbolic.
\medskip 

\noindent
--  Note that the non compact leaves of both examples have exhaustive F{\o}lner sequen\-ces, which combine exhaustive sequences of geodesic balls and suitable sequences of bumps or tubes. An exhaustive F{\o}lner sequence is an increasing F{\o}lner sequence that exhaust the whole space. Kaimanovich gave examples of discrete measured equivalence relations with exhaustive F{\o}lner sequences that are non-amenable (we refer to example  4 in \cite{kaim}).
\end{remarks}

\section{Proofs}

The aim of the this section is to prove that
amenability and F\o lner notions are equivalent for minimal measured foliations. 
Our results may be interpreted as providing an effective version of the amenability criteria established by Kaimanovich \cite{kaimcomptes}  in terms of isoperimetric properties of the leaves: 

\begin{theorem} [Kaimanovich]  A discrete measured equivalence relation $\R$ on a standard Borel space $(T,\mathcal{B},\nu)$ is $\nu$-amenable if and only if for any non-trivial measurable set $T_0$  the induced equivalence relation $\R \mid T_0$ is $\delta$-F\o lner. 
\end{theorem}

In \S~\ref{demostracioncarriereghys} we will rewrite the proof of Carri\`ere and  Ghys' necessary condition, changing the transverse invariant measure for a tangentially smooth measure. In \S~\ref{seccionteoremamedidainvariante}, we will state D. Cass' theorem about minimal leaves, that will allow us to obtain the sufficient condition and to prove theorem \ref{teoremamedidainvariante}. Finally, in \S~\ref{seccionteoremamedidalisa} we will generalize the proof to tangentially smooth measures, thus proving theorem \ref{teoremamedidalisa}.

\subsection{Carri\`ere-Ghys' theorem for tangentially smooth measures}
\label{demostracioncarriereghys}

As we said in the introduction Carri\`ere and Ghys proved that if we
consider a foliation $\mathcal{F}$ of a closed manifold $M$, that is
endowed with a transverse invariant measure $\nu$ and has no essential
holonomy, the amenability of $\mathcal{F}$  implies that it is
F{\o}lner \cite{cagh}. In this section we are going to prove the necessary condition in theorem \ref{teoremamedidainvariante}, for which we do not need the minimality condition, following Carri\`ere-Ghys' method. 

\begin{proposition}
\label{carriereghysquasi-invariante}
Let $\R$ be a discrete measured equivalence relation on a standard Borel space $(T,\mathcal{B},\nu)$. Assume that $\nu$ is a quasi-invariant measure and $\R$ is $\nu$-amenable, then $\nu$-almost every equivalence class is $\delta$-F{\o}lner.
\end{proposition}

\begin{proof}
Since $\R$ is amenable it is hyperfinite, {\it i.e.} there exists an increasing  sequence of finite measured equivalence relations $\R_n$ on $(T, \mathcal{B},\nu)$ such that $\R[x]=\bigcup_n \R_n[x]$ for $\nu$-almost every $x \in T$. Fix $n$, and consider a measurable function $f$ on $T$. Define
$$
\bar{f}(x)=\frac{1}{|\R_n[x]|_x}\sum_{y\in \R_n[x]}f(y)\delta(y,x).
$$
Then, for every $y\in \R_n[x]$ we  have that
$$
\bar{f}(y) = \frac{1}{|\R_n[x]|_y}\sum_{z\in \R_n[x]}f(z)\delta(z,x)\delta(x,y) = 
  \frac{\sum_{z\in \R_n[x]}f(z)\delta(z,x)\delta(x,y)}{\sum_{z\in \R_n[x]}\delta(z,x)\delta(x,y)} = \bar{f}(x).
$$
Hence, we have that 
\begin{eqnarray*}
\int\bar{f}(x)d\nu(x) & =  & \int\frac{1}{|\R_n[x]|_x}f(y)\delta(y,x)d\widetilde{\nu}(x,y) \\
	& = & \int\frac{1}{|\R_n[x]|_x}f(y)\delta(y,x)\delta(x,y)d\widetilde{\nu}(y,x) \\
	& = & \int\frac{f(y)\delta(y,x)\delta(x,y)}{\sum_{z\in \R_n[x]}\delta(z,y)\delta(y,x)}d\widetilde{\nu}(y,x) \\
          &  =  & \int f(y)  \frac{\sum_{x\in \R_n[y]} \delta(x,y)}{\sum_{z\in\R_n[x]}
	\delta(z,y)}d\nu(y) \quad  = \quad  \int f(y) d\nu(y).
\end{eqnarray*}
Denote by $\partial \R_n=\bigcup_{x\in T}\partial \R_n[x]$, and take 
$f_n=\chi_{\partial\R_n}$, the characteristic function. We have that
$$
\bar{f}_n(x) = \frac{1}{|\R_n[x]|_x} \sum_{y\in \R_n[x]}f_n(y)\delta(y,x) = \frac{|\partial \R_n[x]|_x}{|\R_n[x]|_x} 
$$
and therefore
$$
\int \bar{f}_n(x)d\nu(x) =  \int f_n(y) d\nu(y)=\nu(\partial \R_n).
$$
Since $\R[x]=\bigcup_n\R_n[x]$ for $\nu$-almost every $x \in T$, we get that $\nu(\partial \R_n)$ converges to zero as $n\to \infty$. By Fatou's lemma, 
$$
\liminf_{n\to\infty} \bar{f}_n (x)= 
\liminf_{n\to\infty} \frac{|\partial 
\R_n[x]|_x}{|\R_n[x]|_x} = 0
$$
$\nu$-almost everywhere. 
\end{proof}

The previous proposition implies that a foliation $\mathcal{F}$ of a
compact manifold $M$ that is amenable with respect to a
quasi-invariant measure $\nu$  induces a $\delta$-F{\o}lner
equivalence relation on any total transversal.  Since a $\delta$-F\o
lner equivalence class corresponds to an $\eta$-F\o lner leaf, we have
 the following continuous version (that we can also deduce from the hyperfiniteness of the equivalence relation $\R_M$ on the ambient manifold $M$):

\begin{proposition}
\label{carriereghystangencial}
Let $(M,\mathcal{F})$ be a compact foliated manifold endowed with a tangentially smooth measure $\mu$. Assume that $\mathcal{F}$ has no essential holonomy and is $\mu$-amenable. Then $\mathcal{F}$ is $\eta$-F{\o}lner. 
\end{proposition}

\subsection{The invariant measure case: proof of theorem \ref{teoremamedidainvariante}}
\label{seccionteoremamedidainvariante}

In this section we will give a proof of the second implication in theorem \ref{teoremamedidainvariante}, where we will use the minimality condition. We will start by stating a theorem by Cass \cite{cass}, that is a key ingredient in the proof.
As before we will fix a Riemannian metric $g$ on the compact foliated manifold $(M,\mathcal{F})$. Given a point $p\in M$, define $B_M(p,A)$ as the ball with center at $p$ and radius $A$, and $B(p,A)$ the same restricted to the leaf $L_p$ through $p$. We need the following definition.

\begin{definition}\label{defquasiho}
We say that a leaf $L$ of the foliated manifold $(M,\mathcal{F})$ is
\smallskip 

\noindent
-- {\em  recurrent} if for every $\epsilon>0$, there exists $A>0$ such that
$L\subset B_M(B(p,A), \epsilon)$, 
where $A$ depends only on $\epsilon$, and $B_M(B(p,A),\epsilon)$ is the $\epsilon$-neighborhood of the ball $B(p,A)$.
\smallskip 

\noindent
-- {\em quasi-homogeneous} if there exists $k\geq 1$ such that for all $a>0$ and $q\in L$, there exists $A>0$, such that for every ball $B(p,a)$ there exists an immersion $f: B(p,a)\to B(q,A)$, with dilatation bound $k \geq 1$.
\end{definition}

\begin{theorem}[Cass]
\label{teoremahojasminimales}
Let $L$ be a leaf of a foliated compact manifold $(M,\mathcal{F})$. Then $L$ is minimal if and only if one of the following conditions is satisfied:
\smallskip 

\noindent 
(i) $L$ is recurrent;
\smallskip 

\noindent 
(ii) $L$ is quasi-homogeneous.
\end{theorem}

We will not give a proof of this theorem here, we refer the reader to the original paper \cite{cass}. However, we will like to make a remark about the quantifiers involved in the quasi-homogeneity that we will use later. First,  the dilatation of the map $f$ can be taken arbitrarily close to one.  In fact,  the map $f$ is constructed inside a distinguished open set: it maps the ball $B(p,a)$ to a domain contained in $B(q,A)$ that may not contain the point $q$. The dilatation bound implies that
$$
\frac{\mbox{vol}(f(B(p,a))}{\mbox{vol}(B(p,a))} \qquad \mbox{and} \qquad \frac{\mbox{area}(\partial f(B(p,a))}{\mbox{area}(\partial B(p,a))}
$$
are bounded by a constant that depends only on $k$ and the dimension of $\mathcal{F}$.
The main idea is that if we have a minimal leaf without holonomy, then any compact domain in it can be lifted to the nearby leaves with controlled distortion.

\begin{proof}[Proof of Theorem~\ref{teoremamedidainvariante}]
The steps of the proof are the following: first, theorem \ref{teoremahojasminimales} tell us that the leaves are quasi-homogeneous. Secondly, we will construct sequences of probability measures satisfying Reiter's criterion for amenability as in theorem \ref{teoremapromediable}.  
First of all, fix a foliated atlas $\mathcal{A}$ and recall that $d$ is the dimension of $\mathcal{F}$.
Consider a generic leaf $L$. Then, $L$ is a minimal F{\o}lner  leaf. Since it is F{\o}lner, there exist a sequence of compact submanifolds $V_n$ of dimension $d$
such that
$$
\lim_{n\to\infty}\frac{\mbox{area}(\partial V_n)}{\mbox{vol}(V_n)}= 0.
$$
For any point $p\in V_n$ there exist $a_p>0$ such that $V_n \subset B(p,a_p)$. Put $p_n\in V_n$, such that $a_{p_n}=a_n$ is minimal. Take a point $q\in L$, we claim that there exists a sequence of compact submanifolds $W_n^q$ of dimension $d$,
containing $q$, such that
$$
\frac{\mbox{area}(\partial W_n^q)}{\mbox{vol}(W_n^q)}\to 0
$$
as $n\to\infty$. 
To make the notation simpler, we will forget for a moment the index $n$. Beginning with a submanifold $V^p\subset L$, containing the point $p$, we will construct a submanifold $W^q\subset L$ around any given point $q\in L$, whose isoperimetric ratio is comparable to the one of $V^p$. 
By assumption the leaf $L$ is minimal. Using the quasi-homogeneous hypothesis, we have that there exist $A>0$ and an immersion
$f:B(p,a)\to B(q,A)$
into any ball $B(q,A)\subset L$, with  dilatation bound $k \geq 1$.  According to Cass' proof, this dilatation bound $k$ can be taken arbitrarily close to $1$.
Put $V^q=f(V^p)$. Then there exist constants $c$ and $C$, depending only on $k$ and $d$, such that
$$
c \leq \frac{\mbox{vol}(V^q)}{\mbox{vol}(V^p)} \leq C,
$$
and the same is valid for the areas of the boundaries of $V^p$ and $V^q$.
\medskip 

In order to visualize better this construction, we can assume that $B(p,a)$ has trivial holonomy. In this case, $B(p,a)$ can be lifted to nearby leaves, that is, there is a distinguished open set 
$U \cong P\times T$ such that $V^p \subset B(p,a) \subset P\times \{p\}$. On the other hand, 
since $\mathcal{F}$ is minimal, $T$ is a total transversal, so there are constants 
$r \geq a$ and $R$ such that the distance from $q$ to $T \cap L$ is smaller that $R$ and the distance between two points of $T \cap L$ is greater than $r$. Then $A$ may be taken close to $R + a$ and $V^q = f(V^p)$ is a lifted copy of $V^p$ included in the plaque $P \times \{f(p)\} \subset U$ passing through a point $f(p) \in T$ whose distance to $q$ is $\leq R$. 
\medskip 

Returning to the general case, observe that $V^q$ may not contain $q$. To overcome this difficulty let $W^q=V^q \cup P_q$, where $P_q$ is a plaque of the atlas $\mathcal{A}$ containing $q$.
Associated to $V^p$ consider the $d$-current 
$$
\xi^p(\alpha) = \frac{1}{\mbox{vol}(V^p)}\int_{V^p}\alpha,
$$
where $\alpha$ is a differential $d$-form on $M$. 

Let us come back to the F{\o}lner sequence. Starting with the sequence $V_n^p$ 
we obtain the sequence $W_n^q$ such that $q\in W_n^q$. The sequences $V_n^p$ and $W_n^q$ of submanifolds define the sequences of $d$-currents
$$
 \xi_n^p(\alpha) = \frac{1}{\mbox{vol}(V_n^p)}\int_{V_n^p}\alpha \qquad \mbox{and} \qquad 
\xi_n^q(\alpha) =  \frac{1}{\mbox{vol}(W_n^q)}\int_{W_n^q}\alpha,
$$that give rise to sequences of probability measures $\pi_n^p$ and $\pi_n^q$ on $L$, respectively. 
Using Reiter's criterion, we know that for proving the amenability of the foliation $\mathcal{F}$ we need to prove that $\|\pi_n^p-\pi_n^q\|\to 0$, or equivalently that 
$M(\xi_n^p-\xi_n^q)\to 0$,
as $n\to\infty$, where $M$ denotes the mass of a current.  We have that
\begin{eqnarray*}
M(\xi_n^p-\xi_n^q) & \leq & \sup_{ \|\alpha\|=1}\left( 
\frac{1}{\mbox{vol}(V_n^p)}\int_{V_n^p}(Id-f_n^*)\alpha\right) \\
& + & M\left(\frac{|\mbox{vol}(V_n^p)-\mbox{vol}(W_n^q)|}{\mbox{vol}(V_n^p)\mbox{vol}(W_n^q)}\int_{V_n^q}
\alpha\right)
\quad + \quad \frac{\mbox{vol}(P_q)}{\mbox{vol}(W_n^q)} 
\end{eqnarray*}
converge to $0$ as $n\to\infty$. The convergence of the last two terms follows because the volume of $P_q$ is constant and the volume of $V_n^p$ and $W_n^q$ tend to be the same. For the convergence of the first term, though it is not strictly necessary, we can express any $d$-form $\alpha$ as a function multiplied the Riemannian volume on the leaves. The integrals in the first term may now be reduced to integrals of functions.  Assuming again that $V_n^p$ has no holonomy, it is easy to see that the map $f_n : V_n^p \to U_n$ is close to the inclusion map of $V_n^p$ into $U_n$. In general, as in the former case, Cass' theorem provides a map $f_n : V_n^p \to M$ which is  locally (and then globally) close to the inclusion map of $V_n^p$ into $M$. This implies the convergence. Hence, $\|\pi_n^p-\pi_n^q\|\to 0$. The above construction can be done for every point $q \in L$ and for any minimal F{\o}lner leaf $L$ of $\mathcal{F}$. Thus $\mathcal{F}$ satisfies Reiter's criterion, {\it 
 i.e} it is  amenable.
\end{proof}

\begin{remark} 
In fact, Cass' theorem allowed us to remove the holonomy assumption and to prove the theorem for general foliations. On the other hand, if $\mathcal{F}$ has no essential holonomy, we can construct the F{\o}lner sequence $\{W_n^q\}_{n \in \nn}$
without using Cass' theorem. For every point $p$ in a leaf without holonomy, we can assume that each set $V_n^p$ is included in a plaque $P_n \times \{p\}$ of a sequence of  distinguished open sets $U_n \cong P_n \times T_n$ satisfying that $T _{n+1} \subset T_n$. Since $\mathcal{F}$ is minimal, each transversal $T_n \subset U_n$ is a total transversal. Observe that the traces of the leaves are the equivalence classes of an increasing sequence of  equivalence relations $\R_n$ induced by $\mathcal{F}$ on $T_n$. Moreover, if we replace the Riemannian volume by the discrete volume defined in \cite{plan} ({\it i.e.} the volume of a compact set $V$ is the number of vertices of  $V \cap T_0 \subset \R_0[x]$), there is no dilatation when we lift a set transversally into another place of the leaf. In terms of the corresponding graphs, the quasi-homogeneity of the leaves implies that every subgraph of diameter smaller than $ a$ has a copy inside every ball of diameter greater than $ A$. Therefore, it is possible to write the proof from the discrete perspective. This proof is also valid for $C^0$ laminations having no essential holonomy. 
\end{remark}

\subsection{The tangentially smooth measure case: proof of theorem \ref{teoremamedidalisa} }
\label{seccionteoremamedidalisa}

In this section we are concerned with the second implication in theorem \ref{teoremamedidalisa}. 
Here $\mathcal{F}$ is a minimal foliation of a compact manifold $M$,
equipped with a tangentially smooth measure $\mu$ whose generic leaves
are $\eta$-F\o lner with trivial holonomy. 
In order to prove the theorem, we need to deal with the case where $h$ is a measurable function in the transverse direction such that in any distinguished open set $\frac{h}{h(p)}$ is bounded. This is the case when $\mu$ is a harmonic measure 
(see remarks~\ref{hacotada}).  
Recall that $h$ is of class $C^{r-1}$ in the tangent direction, provided that the foliation is of class $C^r$.  

\begin{proof}[Proof of Theorem~\ref{teoremamedidalisa}] Consider a generic leaf $L$ of $\mathcal{F}$, which is minimal, $\eta$-F{\o}lner and has no holonomy. Since $L$ is minimal, it is quasi-homogenous in the sense of definition \ref{defquasiho}. Take an $\eta$-F{\o}lner sequence on $L$, that is, a sequence $\{V_n\}_{n \in \nn}$ of compact domains with boundary such that
$$
\frac{\mbox{area}_h(\partial V_n)}{\mbox{vol}_h(V_n)}\to 0.
$$
As before, since the holonomy of $L$ is trivial, we can suppose that $V_n$ is contained in a plaque of a distinguished open set $U_n=P_n\times T_n$. It is important to observe that the minimality of $L$ implies that the transversal $T_n$ is total.
Let $p_n\in T_n\cap V_n$ we will call $V_n^p=V_n$ to distinguish these
points (we will not write the index for the points $p_n$ to make the
notation easier). For $q\in L$ let 
$$W_n^q=V_n^{f_n(p)}\cup P_q,$$ 
where $f_n$ is the map given in the proof of 
theorem~\ref{teoremamedidainvariante} and $P_q$ is a plaque of $\mathcal{A}$ containing $q$.
Consider the $d$-currents defined by
$$ \xi_n^p(\alpha) = \frac{1}{\mbox{vol}_h(V_n^p)}\int_{V_n^p}\frac{h}{h(p)}\alpha \qquad \mbox{and} \qquad 
\xi_n^q(\alpha) = \frac{1}{\mbox{vol}_h(W_n^q)}\int_{W_n^q}\frac{h}{h(f_n(p))}\alpha,
$$
for any differential $d$-form $\alpha$. They will be identified with
the corresponding probability measures on $L$. According to the proof in section \ref{seccionteoremamedidainvariante}, we have to prove that the mass of the difference of the two integral currents converges to zero, to conclude that $\mathcal{F}$ is $\mu$-amenable. 
\medskip 

For each distinguished open set $U_ n \cong P_n \times T_n$, we have that
$$
\mu(U_n)  =  \int_{T_n} \int_{P_n \times\{y\}}h(x,y)d\mbox{vol}^y(x,y)d\nu(y)  <  + \infty.
$$
This implies that the function $h|_{U_n}$ is integrable for the
measure $d\mbox{vol}^y \otimes d\nu$ and  its $L_1$-norm is equal
to $\mu(U_n) < + \infty$. Thus $h|_{U_n}$ can be
approximated in the $L^1$-norm (with respect to the measure $d\mbox{vol}^y \otimes d\nu$) by a monotone increasing sequence of continuous functions $h_{n,m}$ with compact support in $U_n$, that is
$$
 \|h|_{U_n}-h_{n,m}\|_1 = \int_{T_n} \int_{P_n}|h| _{U_n}-h_{n,m}|dvol^yd\nu(y) \to 0
$$
when $m\to \infty$. Moreover, since the function
$\frac{h}{h(p)}$ is bounded in any distinguished open set $U_n$, it
is also integrable. 
Hence, for each $U_n$, we can approximate  $\frac{h}{h(p)}$ in
the $L^1$-norm by a monotone increasing sequence of continuous
functions of the form $\frac{h_{n,m}}{h_{n,m}(p)}$ with compact support in $U_n$. Now,  for every pair $(n,m)$, we have that
$\Big\|  \frac{h}{h(p)} - \frac{h_{n,m}}{h_{n,m}(p)} \Big\|_1$ is uniformly bounded.
\medskip 

Write $p=(x,y)$, $f_n(p)=(x,y^\prime)$ and $q=(u,v)$ where the first
coordinate is in the plaque and the second one in the transversal, and
we omitted the indexes for $p$ and $(x,y)$. Let us begin by fixing the indexes $n, m$ and expressing the difference of the two currents:
\begin{eqnarray*}
\xi_n^p-\xi_n^q & = & \frac{1}{\mbox{vol}_h(V_n^p)}\int_{V_n^p}\frac{h}{h(p)}\alpha       
     -\frac{1}{\mbox{vol}_h(W_n^q)}\int_{W_n^q}\frac{h}{h(f_n(p))}\alpha\\
& = & \frac{1}{\mbox{vol}_h(V_n^p)}\int_{V_n^p}\frac{h}{h(p)}\alpha
     -\frac{1}{\mbox{vol}_h(V_n^p)}\int_{V_n^p}\frac{h_{n,m}}{h_{n,m}(p)}\alpha \\
& + & \frac{1}{\mbox{vol}_h(V_n^p)}\int_{V_n^p}\frac{h_{n,m}}{h_{n,m}(p)}\alpha
     -\frac{1}{\mbox{vol}_h(W_n^q)}\int_{V_n^{f_n(p)}}\frac{h_{n,m}}{h_{n,m}(f_n(p))}\alpha\\
& + & \frac{1}{\mbox{vol}_h(W_n^q)}\int_{V_n^{f_n(p)}}\frac{h_{n,m}}{h_{n,m}(f_n(p))}\alpha
     -\frac{1}{\mbox{vol}_h(W_n^q)}\int_{V_n^{f_n(p)}}\frac{h}{h(f_n(p))}\alpha \\
& - &  \frac{1}{\mbox{vol}_h(W_n^q)}\int_{P_q}\frac{h}{h(f_n(p))}\alpha.
\end{eqnarray*}
Let
$$
\xi^p_{n,m} =  \frac{1}{\mbox{vol}_h(V_n^p)}\int_{V_n^p}\frac{h_{n,m}}{h_{n,m}(p)}\alpha \qquad \mbox{and} \qquad 
\xi^{f_n(p)}_{n,m} =  \frac{1}{\mbox{vol}_h(W_n^q)}\int_{V_n^{f_n(p)}}\frac{h_{n,m}}{h_{n,m}(f_n(p))}\alpha.
$$
Hence, in the $L_1$-norm, we have that
\begin{eqnarray*}
\|\xi_n^p-\xi_n^q\|_1 & \leq & \frac{1}{\mbox{vol}_h(V_n^p)}\int_{V_n^p}\Big|\frac{h|_{U_n}}{h(p)}-\frac{h_{n,m}}{h_{n,m}(p)}\Big|dvol^y+\|\xi^p_{n,m}-\xi_{n,m}^{f_n(p)}\|_1\\
& &  +\frac{1}{\mbox{vol}_h(W_n^q)}\int_{V_n^{f_n(p)}}\Big|\frac{h| _{U_n}}{h(f_n(p))}-\frac{h_{n,m}}{h(f_n(p))}\Big|dvol^{y^\prime} \\
& + & \frac{1}{\mbox{vol}_h(W_n^q)}\int_{P_q} \frac{h}{h(f_n(p))} dvol^v,
\end{eqnarray*}
where the last term is equal to $\mbox{vol}_h(P_q)/\mbox{vol}_h(W_n^q)$. 
Since $\frac{h}{h(f_n(p))}$ is bounded, we can assume that $\mbox{vol}_h(P_q)$ is bounded 
hence this term goes to zero when $n\to \infty$.
\medskip 

To show that $\mathcal{F}$ is amenable we are going to use again Reiter's criterion (but this time according to the formulation given in theorem \ref{teoremapromediable} {\it (iii)}  adapted to the continuous setting). Let $\gamma_n$ be the holonomy transformation that maps the point $p\in L$ to 
$q \in L$. We have that
 \begin{eqnarray*}
& &  \lim_{n\to\infty} \int_{\mbox{Dom}(\gamma_n)}\|\xi_n^p-\xi_n^q\|_1d\nu(y)  \\
& \leq & \lim_{n\to\infty} \int_{T_n} \frac{1}{\mbox{vol}_h(V_n^p)}  \int_{V_n^p}\Big|\frac{h| _{U_n}}{h(p)}-\frac{h_{n,m}}{h_{n,m}(p)}\Big|dvol^yd\nu(y) \nonumber \\
& +  & \lim_{n\to\infty} \int_{T_n}\|\xi^p_{n,m}-\xi_{n,m}^{f_n(p)}\|_1d\nu(y)\\
& + &  \lim_{n\to\infty} \int_{T_n} \frac{1}{\mbox{vol}_h(V_n^{f_n(p)})} 
\int_{V_n^{f_n(p)}}\Big|\frac{h| _{U_n}}{h(f_n(p))}-\frac{h_{n,m}}{h_{n,m}(f_n(p))}\Big|dvol^{y^\prime}d\nu(y^\prime). \nonumber
\end{eqnarray*}
Let us see what happens with each term in the sum when $n$ goes to infinity:  
\smallskip 

\noindent
-- There exists $c \geq 0$ such that $\Big\|\frac{h|_{U_n}}{h(p)}-\frac{h_{n,m}}{h_{n,m}(p)}\Big\|_1\leq c$ for every pair $n,m \in \nn$. On the other hand, there exists $C_n \leq \mbox{vol}_h(V_n^p)$ such that $C_n \to  \infty$ when $n$ goes to $\infty$. For each $m \in \nn$, we have that 
\begin{eqnarray*}
& & \lim_{n\to\infty} \int_{T_n}\frac{1}{\mbox{vol}_h(V_n^p)}\int_{V_n^p}\Big|\frac{h| _{U_n}}{h(p)}-\frac{h_{n,m}}{h_{n,m}(p)}\Big|dvol^yd\nu(y)  \\ 
& \leq &  \lim_{n\to\infty} \frac{1}{C_n} \Big\|\frac{h|_{U_n}}{h(p)}-\frac{h_{n,m}}{h_{n,m}(p)}\Big\|_1 \quad \leq \quad \lim_{n\to\infty} \frac{c}{C_n} \quad =  \quad 0.
\end{eqnarray*}

\noindent
-- The same argument applies to the third therm in the sum above.
\smallskip 

\noindent
-- Since the functions $h_{n,m}$ are continuous with compact support, the integral currents $\xi^p_{n,m}$ and $\xi^{f_n(p)}_{n,m}$ behave as the integral currents in the proof in section \ref{seccionteoremamedidainvariante}, hence $\|\xi^p_{n,m}-\xi^{f_n(p)}_{n,m}\|\to 0$ as $n\to \infty$.
\smallskip 

\noindent
Therefore, for $m$ big enough, 
$$\lim_{n\to \infty} \int_{\mbox{Dom}(\gamma_n)}\|\xi_n^p-\xi_n^q\|_1d\nu(y)=0,$$
implying that the foliation $\mathcal{F}$ is $\mu$-amenable. 
\end{proof}

\section{Final comments}
\label{final}

Before finishing, we will like to point out some questions related to this work that we believe are important. Some of them have already appear in the literature, or were posed to us. In some cases we have a partial answer, and in some others we will like to motivate the reader.
\medskip 

\begin{list}{\labelitemi}{\leftmargin=0pt}
\item[--] {\em Measurable version of theorem
    \ref{teoremamedidainvariante}}
    
The existence of global and leafwise means (which follows from the
amenability and F\o lner conditions) are measurable properties. But we
have used the minimality, a property of topological nature, in order
to prove the equivalence between this notions. D. Gaboriau asked us if
there is a measurable version
of theorem \ref{teoremamedidainvariante}. We can give a partial 
answer. 
\medskip 

Assume that $\mathcal{R}$ is a discrete measured equivalence relation on a standard Borel space $T$, endowed with an 
invariant measure $\nu$. As in the case of 
foliations without  holonomy, finite sets in the equivalence classes are stacked in Borel sets isomorphic to the product of finite sets and Borel subsets of $T$. 
If $\nu$ is {\em conservative}, then there is a decreasing sequence of Borel sets $T_n\subset T$ which meet almost every equivalence class such that $\bigcap_{n \in \nn} T_n = \emptyset$. We can replace minimality by this property.  However,  
we must fix a {\em graphing} $\Phi$ in the sense of \cite{gaboriau} (or equivalently a {\em graph structure} in the sense of \cite{kaimcomptes}) to endow each equivalence class with a metric structure. 
For this,  we may suppose that  all the equivalence relations $\mathcal{R} |_{T_n}$ are Kakutani equivalent (with respect to a sequence of graphings induced by $\Phi$ according to a construction in \cite{gaboriau}, 
\S~II.B). This means that for each $n \in \nn$ and for each point $x \in T_{n+1}$, the inclusion of  $\mathcal{R} |_{T_{n+1}}[x] $ into $\mathcal{R} |_{T_n} [x]$ is a quasi-isometry.
We can resume this discussion in the following result:

\begin{theorem} Let $\mathcal{R}$ be a discrete measured equivalence
  relation on a standard Borel space $T$, endowed with an invariant measure $\nu$ with Radon-Nikodym cocycle $\delta$ and a graphing $\Phi$. Assume that $\nu$ is conservative and therefore there exists a decreasing sequence of Borel sets $T_n\subset T$ which meet almost every equivalence class such that $\bigcap_{n \in \nn} T_n = \emptyset$, and that  the graphed equivalence relations $\mathcal{R} |_{T_n}$ and $\mathcal{R} |_{T_{n+1}}$ are Kakutani equivalent.
Then $\mathcal {R}$ is $\nu$-amenable if and only if 
$\mathcal {R}$ is $\delta$-F\o lner.
\end{theorem}

 \item[--] {\em Averaging sequences and tangentially smooth measures}
 
From Plante's work \cite{plan}, it is known that usual F\o lner sequences define transverse invariant measures supported by the limit set of the sequence.
Borrowing the terminology introduced in \cite{goodmanplante}, we will say that any $\eta$-F\o lner sequence $\{V_n\}_{n \in \nn}$ with
$\mbox{area}_h(\partial V_n)/\mbox{vol}_h(V_n)\to 0$
defines an {\em averaging sequence}
$$
 \mu_n(f)  =  \frac{1}{\mbox{vol}_h(V_n)}\int_{V_n} \frac{h(p)}{h(p_n)} \, f(p) \, dvol(p)  \ \ , \ \ \forall f \in C(M),
$$
where $p_n\in V_n$.
Taking a subsequence if necessary, we may assume that 
the averaging sequence
$\mu_n$ converge weakly to 
a measure $\mu$. 
A natural question is to ask if the measure $\mu$  is tangentially smooth. We do not know the 
answer in general, but we know that {\em harmonic averaging sequences
define harmonic measures.} Details will be published in a future paper.
\medskip 

As we previously said, the F{\o}lner and amenable conditions depend strongly on the measure we are considering. Since a F{\o}lner leaf gives us a transverse invariant measure, we could ask (as suggested by S. Hurder) if any foliation is amenable with respect to such a measure.
Observe that the two examples constructed in this paper are amenable with respect to  the atomic measures (supported by the toric leaves) obtained as limit of averaging sequences.
\medskip 

\item[--] {\em F{\o}lner and Liouvillian foliations}

A stronger condition than amenability for a foliation $\mathcal{F}$ of
a compact manifold is to be  {\it Liouvillian}, we refer to proposition
20 of \cite{cofw}. If $\mathcal{F}$ is
equipped with a harmonic measure $\mu$, we say that it 
is  {\em Liouvillian} if $\mu$-almost every leaf does not admit non-constant
bounded harmonic functions. In \cite{kaimbrowniano}, Kaimanovich
proved that a foliation with sub-exponential growth is Liouvillian.
\medskip  

Observe that the examples we constructed imply that there
are F{\o}lner  foliations that are not Liouvillian. Nonetheless, we
could ask if a minimal F{\o}lner foliation is Liouvillian. A positive
answer to this question would imply that every harmonic measure of a minimal F{\o}lner foliation is
completely invariant 
(see \cite{kaimbrowniano}).
\medskip 

\item[] {\em -- F{\o}lner foliations are generically amenable}

The examples we gave of non-amenable F{\o}lner foliations are
unstable  because a small perturbation looses the F{\o}lner
condition. Ghys asked us if {\it almost all, or for an open dense set of foliations,}  F{\o}lner foliations are amenable. The sense of this {\it almost all} has to be properly defined in the set of F{\o}lner foliations. 
\medskip 

As a first observation, we know that the existence of a F\o lner sequence implies the existence of a transverse invariant measure. According to a theorem of D. Sulli\-van (see \cite{sull1}), there is an one-to-one correspondence between transverse invariant measures and foliated cycles. Since foliated cycles form a closed cone in the space of  foliated currents, the existence of a transverse invariant measure for a given foliation is a closed condition. Hence, being F{\o}lner is also a closed condition. But we ignore if the set of F\o lner foliations has empty interior.
\end{list}

\bibliographystyle{plain} 

\bibliography{biblioDCDS}

\end{document}